\documentclass[10pt]{amsart}

\usepackage{amsmath,amsfonts,latexsym,amssymb,mathrsfs,setspace,enumerate,color,cite,graphicx,epsf,qtree,amscd,fullpage,float}

\newtheorem{thm}{Theorem}
\newtheorem{lem}[thm]{Lemma}
\newtheorem{prop}[thm]{Proposition}
\newtheorem{cor}[thm]{Corollary}
\numberwithin{equation}{section}
\numberwithin{thm}{section}

\theoremstyle{definition}
\newtheorem{ex}[thm]{Example}
\newtheorem{conj}[thm]{Conjecture}

\newcommand{\rat}{\mathbb Q}
\newcommand{\real}{\mathbb R}

\newcommand{\alg}{\overline\rat}
\newcommand{\algt}{\alg^{\times}}

\newcommand{\nat}{\mathbb N}

\newcommand{\tors}{\mathrm{tors}}
\newcommand{\zz}{{\bf z}}
\newcommand{\xx}{{\bf x}}
\newcommand{\yy}{{\bf y}}

\newcommand{\NN}{\mathcal U}
\newcommand{\DD}{\mathcal L}

\newcommand{\GG}{\mathcal G}

\newcommand{\convex}{\mathrm{Conv}}
\newcommand{\vertex}{\mathrm{Vert}}

\newcommand{\comment}[1]{}

\title[Continued Fractions and the Metric Mahler Measure]{Continued Fraction Expansions in Connection with the Metric Mahler Measure}

\author{Charles L. Samuels}

\address{Christopher Newport University, Department of Mathematics, 1 Avenue of the Arts, Newport News, VA 23606}
\email{charles.samuels@cnu.edu}

\begin{document}

\begin{abstract}
	The metric Mahler measure was first studied by Dubickas and Smyth in 2001 as a means of phrasing Lehmer's conjecture in topological language.  
	More recent work of the author examined a parametrized family of generalized metric Mahler measures that gives rise to a series of new, and apparently difficult, problems.  
	We establish a connection between these metric Mahler measures and the theory of continued fractions in a certain class of special cases.  
	Our results enable us to calculate metric Mahler measures in several new examples.
\end{abstract}

\maketitle

\section{Introduction}

Suppose that $K$ is a number field and $v$ is a place of $K$ dividing the place $p$ of $\rat$.  Let $K_v$ and $\rat_p$ be their respective completions
so that $K_v$ is a finite extension of $\rat_p$.  We note the well-known fact that
\begin{equation*}
	\sum_{v\mid p} [K_v:\rat_p] = [K:\rat],
\end{equation*}
where the sum is taken over all places $v$ of $K$ dividing $p$.  Given $x\in K_v$, we define $\|x\|_v$ to be the unique extension of the $p$-adic absolute value on $\rat_p$ and set
\begin{equation} \label{NormalAbs}
	|x|_v = \|x\|_v ^{[K_v:\rat_p]/[K:\rat]}.
\end{equation}
If $\alpha\in K$, then $\alpha\in K_v$ for every place $v$, so we may define the {\it (logarithmic) Weil height} by
\begin{equation*}
	h(\alpha) = \sum_v \log^+ |\alpha|_v.
\end{equation*}
Due to our normalization of absolute values \eqref{NormalAbs}, this definition is independent of $K$, meaning that $h$ is well-defined
as a function on the multiplicative group $\algt$ of non-zero algebraic numbers.

It is well-known that $h(\alpha) = 0$ if and only if $\alpha$ is a root of unity, and it can easily be verified that $h(\alpha^n) = |n|\cdot h(\alpha)$ for all
integers $n$.  In particular, we see that $h(\alpha) = h(\alpha^{-1})$. If $\alpha\in \rat^\times$ there exists relatively prime integers $r$ and $s$ such that $\alpha = r/s$.
Under these assumptions, we have that $h(\alpha) = \log \max\{|r|,|s|\}$.

The Weil height is closely connected to a famous 1933 problem of D.H. Lehmer \cite{Lehmer}.  The {\it (logarithmic) Mahler measure} of a non-zero algebraic number $\alpha$ is
defined by 
\begin{equation} \label{MahlerMeasure}
	m(\alpha) = [\rat(\alpha):\rat] \cdot h(\alpha).  
\end{equation}	
In attempting to construct large prime numbers, Lehmer came across the problem of determining whether there exists a sequence of
algebraic numbers $\{\alpha_n\}$, not roots of unity, such that $m(\alpha_n)$ tends to $0$ as $n\to\infty$.  This problem remains unresolved, although substantial evidence
suggests that no such sequence exists (see \cite{BDM, MossWeb, Schinzel, Smyth}, for instance).  This assertion is typically called Lehmer's conjecture.

\begin{conj}[Lehmer's Conjecture]
	There exists $c>0$ such that $m(\alpha) \geq c$ whenever $\alpha\in \algt$ is not a root of unity. 
\end{conj}

Dobrowolski \cite{Dobrowolski} provided the best known lower bound on $m(\alpha)$ in terms of $\deg\alpha$, while Voutier \cite{Voutier}
later gave a  version of this result with an effective constant.  Nevertheless, only little progress has been made on Lehmer's conjecture for an arbitrary algebraic number $\alpha$.

Dubickas and Smyth \cite{DubSmyth, DubSmyth2} were the first to study a modified version of the Mahler measure that has the triangle inequality.
A point $(\alpha_1,\alpha_2,\ldots,\alpha_N)\in (\algt)^N$ is called a {\it product representation of $\alpha$} if $\alpha = \prod_{n=1}^N \alpha_n$,
and we write $\mathcal P(\alpha)$ to denote the set of all product representations of $\alpha$.  Dubickas and Smyth defined the {\it metric Mahler measure} by
\begin{equation*} \label{m1}
	m_1(\alpha) = \inf\left\{ \sum_{n=1}^N m(\alpha_n): (\alpha_1,\alpha_2,\ldots,\alpha_N)\in \mathcal P(\alpha)\right\}.
\end{equation*}
It is verified in \cite{DubSmyth2} that $(\alpha,\beta) \mapsto m_1(\alpha\beta^{-1})$ is a well-defined metric on $\algt/\algt_\tors$ which
induces the discrete topology if and only if Lehmer's conjecture is true.
The author \cite{SamuelsCollection, SamuelsParametrized, SamuelsMetric} extended the metric Mahler measure to the {\it $t$-metric Mahler measure}
\begin{equation} \label{mt}
	m_t(\alpha) = \inf\left\{ \left( \sum_{n=1}^N m(\alpha_n)^t\right)^{1/t}: (\alpha_1,\alpha_2,\ldots,\alpha_N)\in \mathcal P(\alpha) \right\}
\end{equation}
which is well-defined for all $t > 0$.
In this context, we examined the function $t\mapsto m_t(\alpha)$ for a fixed algebraic number $\alpha$.  For instance, we showed that this function is everywhere continuous 
and infinitely differentiable at all but finitely many points.

As $\mathcal P(\alpha)$ is infinite, if one is studying $m_t(\alpha)$, it is often helpful to be able to replace $\mathcal P(\alpha)$ with a finite set in \eqref{mt}.
A finite subset $\mathcal X(\alpha) \subseteq \mathcal P(\alpha)$ is called an {\it infimum set for $\alpha$} if 
\begin{equation*}
	m_t(\alpha) = \min\left\{ \left( \sum_{n=1}^N m(\alpha_n)^t\right)^{1/t}: (\alpha_1,\alpha_2,\cdots,\alpha_N)\in \mathcal X(\alpha)\right\}
\end{equation*}
for all $t > 0$.  As part of an article examining general metric heights, the author \cite{SamuelsMetric} established the following.

\begin{thm}[S., 2014] \label{FiniteAttain}
	Every algebraic number has an infimum set.
\end{thm}

In view of this result, we may conclude that $t\mapsto m_t(\alpha)$ is a piecewise function with finitely many pieces each having the form
\begin{equation*}
	t\mapsto  \left( \sum_{n=1}^N m(\alpha_n)^t\right)^{1/t},
\end{equation*}
where $(\alpha_1,\alpha_2,\ldots,\alpha_N)\in \mathcal P(\alpha)$.  In particular, there exists a product representation of $\alpha$ which attains the infimum in $m_t(\alpha)$
for all sufficiently large $t$.  

While Theorem \ref{FiniteAttain} has these useful consequences, it also has a key weakness -- it provides no method for constructing a particular
infimum set for $\alpha$.  The earlier work of Jankauskas and the author \cite{JankSamuels}, although it applies only to the special case $\alpha\in \rat$, 
does not suffer from this same disadvantage.  

Suppose without loss of generality that $\alpha$ is a positive rational number with $\alpha\ne 1$.  A point
\begin{equation*} \label{FactorDef}
	\left( \frac{r_1}{s_1},\frac{r_2}{s_2},\cdots,\frac{r_N}{s_N}\right) \in \mathcal P(\alpha)
\end{equation*}
is called a {\it factorization} of $\alpha$ if 
\begin{enumerate}[(i)]
	\item $r_n,s_n \in \nat$ for all $1\leq n\leq N$
	\item $r_n/s_n \ne 1$ for all $1\leq n\leq N$
	\item $\gcd(r_m,s_n) = 1$ for all $1\leq m,n\leq N$
\end{enumerate}
If we further assume that $r$ and $s$ are relatively prime positive integers with $\alpha = r/s$, then
\begin{equation*}
	r = \prod_{n=1}^N r_n\quad\mbox{and}\quad s = \prod_{n=1}^N s_n.
\end{equation*}
This implies that the set $\mathcal F(\alpha)$ of factorizations of $\alpha$ is finite and we obtain the following result.

\begin{thm}[Jankauskas \& S., 2012] \label{RationalAttain}
	If $\alpha$ is a positive rational number with $\alpha\ne 1$ then $\mathcal F(\alpha)$ is an infimum set for $\alpha$.
\end{thm}

In the special case $\alpha \in \rat$, this result is an improvement over Theorem \ref{FiniteAttain}, which only establishes the existence of an infimum set for $\alpha$,
whereas Theorem \ref{RationalAttain} identifies a specific infimum set. 
In this article, we further specialize to the case where $\alpha$ has exactly two distinct prime factors.  That is, we assume that $\alpha = p^a/q^b$, 
where $p$ and $q$ are primes and $a,b\in \nat$.  Our main results, Theorems \ref{FirstAxioms} and \ref{Main}, produce a particular infimum set
arising from the continued fraction expansion for $\log q/\log p$.  In spite of the specialized nature of our results, we believe they have the following two advantages:
\begin{enumerate}
	\item Our infimum set $\mathcal B(\alpha)$ is smaller than $\mathcal F(\alpha)$, and hence, we improve the best known result (Theorem \ref{RationalAttain})
		in the special case where $\alpha = p^a/q^b$.
	\item\label{ApproxConnection} We establish a connection between the metric Mahler measures and the theory of continued fractions. 
\end{enumerate}
We believe that \eqref{ApproxConnection} is of particular interest because it is the first known connection between the metric Mahler measures and a more classical 
area of number theory.

The remainder of this article is structured as follows.  In Section \ref{BestApproximations}, we provide two key definitions --  upper and lower best rational approximations
to an irrational number $\xi$.  
Further, we apply a well-known technique (see \cite[pp. 55-63]{Perron}) to note these rational numbers are obtained via continued fraction expansions.  
In Section \ref{Mains}, we state our two main results showing that the infimum in $m_t(\alpha)$ is attained using the aforementioned upper and lower best approximations to
$\log q/\log p$.  In Section \ref{Apps}, we consider several new applications of our main results, we establish an estimate on the size of $\mathcal B(\alpha)$ in certain special cases,
and we pose some open questions arising from our work.  Finally, we provide the proofs of all of our new results in Section \ref{Proofs}.

\section{Upper and lower best approximations} \label{BestApproximations}

We shall write $\nat_0 = \{0,1,2,3,\ldots\}$ and let $\mathcal N = \{(a,b)\in \nat_0\times\nat_0: (a,b)\ne (0,0)\}$.
If $(a,b),(c,d)\in \mathcal N$ we define addition of these elements in the natural way
\begin{equation} \label{AdditionND}
	(a,b) + (c,d) = (a+c,b+d).
\end{equation}
Suppose that $\xi$ is a positive real irrational number and define the {\it upper} and {\it lower sets for $\xi$} by
\begin{equation*}
	\NN(\xi) = \left\{(a,b)\in \mathcal N: a > b\xi\right\}\quad\mbox{and}\quad\DD(\xi) = \left\{(a,b)\in \mathcal N: a < b\xi\right\}.
\end{equation*}
It is clear from these definitions that both $\NN(\xi)$ and $\DD(\xi)$ are closed under the addition defined in \eqref{AdditionND}, and moreover, 
$\mathcal N = \NN(\xi) \cup \DD(\xi)$ is a disjoint union.
For the purposes of this article, we adopt the convention that $\gcd(a,0) = a$ and that $a/0 = \infty$ for all $a\in \nat$.  As a result, we obtain that
\begin{equation*}
	\NN(\xi) = \left\{(a,b)\in \mathcal N: \frac{a}{b} > \xi\right\} \quad\mbox{and}\quad\DD(\xi) = \left\{(a,b)\in \mathcal N: \frac{a}{b} < \xi\right\}.
\end{equation*}
Additionally, we write $\GG = \{(a,b)\in \mathcal N: \gcd(a,b) = 1\}$.  Now define the subsets of $\NN_1(\xi)$ and $\NN_2(\xi)$ of $\NN(\xi)$ by
\begin{equation*}
	\NN_1(\xi) = \left\{ (a,b)\in \NN(\xi): \frac{a}{b} = \min\left\{ \frac{m}{n}: (m,n)\in \NN(\xi),\ m \leq a\right\}\right\}
\end{equation*}
and
\begin{equation*}
	\NN_2(\xi) = \left\{ (a,b)\in \NN(\xi): \frac{a}{b} = \min\left\{ \frac{m}{n}: (m,n)\in \NN(\xi),\ n\leq b\right\}\right\}.
\end{equation*}
Roughly speaking, we may think of the points in $\NN_1(\xi)$ and $\NN_2(\xi)$ as best approximations to $\xi$ from above.
Similarly, we define the analogous subsets $\DD_1(\xi)$ and $\DD_2(\xi)$ of $\DD(\xi)$ by
\begin{equation*}
	\DD_1(\xi) = \left\{ (a,b)\in \DD(\xi): \frac{a}{b} = \max\left\{ \frac{m}{n}: (m,n)\in \DD(\xi),\ m \leq a\right\}\right\}
\end{equation*}
and
\begin{equation*}
	\DD_2(\xi) = \left\{ (a,b)\in \DD(\xi): \frac{a}{b} = \max\left\{ \frac{m}{n}: (m,n)\in \DD(\xi),\ n\leq b\right\}\right\}.
\end{equation*}
Analogous to the definitions of $\NN_1(\xi)$ and $\NN_2(\xi)$, we interpret the points in $\DD_1(\xi)$ and $\DD_2(\xi)$ as best approximations to $\xi$ from below.
It is straightforward to make the following observations regarding the above sets.

\begin{prop} \label{NDBasics}
	Suppose that $\xi$ is a positive real irrational number.
	\begin{enumerate}[(i)]
		\item $(a,b)\in \NN(\xi)$ if and only if $(b,a)\in \DD(\xi^{-1})$
		\item $(a,b)\in \NN_1(\xi)$ if and only if $(b,a)\in \DD_2(\xi^{-1})$
		\item $(a,b)\in \NN_2(\xi)$ if and only if $(b,a)\in \DD_1(\xi^{-1})$
	\end{enumerate}
\end{prop}

We also observe that, in general, we do not have $\NN_1(\xi) = \NN_2(\xi)$ or $\DD_1(\xi) = \DD_2(\xi)$.  For instance, if $0<\xi < 1/2$ then $(1,1)\in \NN_2(\xi)$ but $(1,1)\not \in \NN_1(\xi)$.  
However, we do obtain the following relationships among these sets.

\begin{prop}\label{Relations}
	Suppose that $\xi$ is a positive real irrational number.
	\begin{enumerate}[(i)]
		\item\label{Contains} $\NN_1(\xi)\subseteq \NN_2(\xi)$ and $\DD_2(\xi) \subseteq \DD_1(\xi)$
		\item\label{UpperEqual} If $\xi > 1$ then $\GG\cap \NN_1(\xi) = \GG\cap \NN_2(\xi)$
		\item\label{LowerEqual} If $\xi < 1$ then $\GG\cap\DD_1(\xi) = \GG\cap\DD_2(\xi)$
	\end{enumerate}
\end{prop}

Each element of $\GG\cap\NN_1(\xi)$ is called an {\it upper best approximation for $\xi$} and each element of $\GG\cap\DD_2(\xi)$ is called a {\it lower best approximation for $\xi$.}
As a cautionary note, our definition of upper best approximation is slightly more restrictive than what is normally seen in the literature (see \cite{Kimberling,Perron}, for instance).
Typically, the elements of $\GG\cap\NN_2(\xi)$ are called the upper best approximations to $\xi$.  However, in view of Proposition \ref{Relations}\eqref{Contains}, we find
our definition to be best suited to our purposes.  Nevertheless, under the assumption that $\xi > 1$, Proposition \ref{Relations}\eqref{UpperEqual} shows that our definitions of 
upper and lower best approximations are equivalent to the more standard definitions.  In this situation, there is a well-known method described in \cite[pp. 55--63]{Perron} 
(or see \cite[p. 123]{Kimberling} for a more brief discussion) for listing the upper and lower best approximations for $\xi$ which we summarize as a theorem.

\begin{thm} \label{ApproxFrac}
	Suppose that $\xi > 1$ is an irrational number having continued fraction expansion given by $\xi = [x_0;x_1,x_2,x_3,\ldots]$ and let $(a,b)\in \mathcal N\setminus \{(1,0),(x_0,1)\}$.  
	Then $(a,b)$ is an upper or lower best approximation to $\xi$ if and only if $\gcd(a,b) = 1$ and  there exists $n\in \nat$ and $1\leq x \leq x_n$ such that
	$a/b = [x_0;x_1,x_2,\ldots,x_{n-2},x_{n-1},x]$.
\end{thm}

For comparison purposes, it is worth noting the connection between the above definitions and the definition of best approximation.  An ordered pair $(a,b)$ is called
a {\it best approximation\footnote{A rational number satisfying this definition is often called {\it best approximation to $\xi$ of the first kind}.  Indeed, there is a commonly used
definition for {\it best approximation of the second kind} which is more restrictive.  However, we believe the best approximation of the first kind is the proper analog of upper and 
lower best approximations as defined in this paper.  Hence, we only mention its definition here.} to $\xi$} if
\begin{enumerate}[(i)]
	\item $\gcd(a,b) = 1$
	\item If $|\xi - a/b| > |\xi - r/s|$ for some $(r,s)\in \mathcal N$ then $s > b$.
\end{enumerate}
If $(a,b)$ is a best approximation to $\xi>1$ then it follows that $(a,b)$ is either an upper best approximation or a lower best approximation to $\xi$.
However, the converse of this statement is false.  For instance, if we take $\xi = \log 3/\log 2 = 1.5849625\ldots$ then it can be shown that $(27,17)$ is an upper best approximation to $\xi$.  
On the other hand,
\begin{equation*}
	\left| \xi - \frac{27}{17}\right| = 0.00327279\ldots\quad\mbox{and}\quad \left| \xi - \frac{19}{12}\right| = 0.00162917\ldots
\end{equation*}
meaning that $(27,17)$ is not a best approximation to $\xi$.  Nevertheless, the upper and lower best approximations -- not the best approximations -- are
the relevant definitions for our mains results.

\section{Main results} \label{Mains}

As noted above, this article is devoted to studying the infimum sets for $\alpha = p^a/q^b$.  In this situation, every factorization of $\alpha$ has the form
\begin{equation*}
	\left( \frac{p^{a_1}}{q^{b_1}}, \frac{p^{a_2}}{q^{b_2}},\ldots, \frac{p^{a_N}}{q^{b_N}}\right),
\end{equation*}
where $a = a_1 + a_2 + \cdots + a_N$ and $b = b_1 + b_2 + \cdots + b_N$.  The first of our two main results considers the general case $(a,b)\in \mathcal N$.

\begin{thm} \label{FirstAxioms}
	Let $p$ and $q$ be distinct primes and set $\xi = \log q/\log p$.  Assume that $(a,b)\in \mathcal N$ and write $\alpha = p^a/q^b$.  If
	\begin{equation*}
		\left( \frac{p^{a_1}}{q^{b_1}}, \frac{p^{a_2}}{q^{b_2}},\ldots, \frac{p^{a_N}}{q^{b_N}}\right)
	\end{equation*}
	is a factorization of $\alpha$ which attains the infimum in $m_t(\alpha)$ for some $t > 1$ then, for every $1\leq n \leq N$,
	$(a_n,b_n)\in \GG\cap\NN_2(\xi)$ or $(a_n,b_n)\in \GG\cap\DD_1(\xi)$.
\end{thm}

Unfortunately, we cannot conclude from Theorem \ref{FirstAxioms} that $(a_n,b_n)$ is an upper or lower best approximation to $\xi$. 
That is, we cannot conclude that $(a_n,b_n)\in \GG\cap\NN_1(\xi)$ or that $(a_n,b_n)\in \GG\cap\DD_2(\xi)$.  
Indeed, consider for instance $\alpha = 8/25 = 2^3/5^2$ so that $\xi = \log 5/\log 2$.  In this case, it can be shown that
\begin{equation} \label{ExampleFail}
	\left( \frac{4}{5}, \frac{2}{5}\right) = \left( \frac{2^2}{5^1}, \frac{2^1}{5^1}\right) 
\end{equation}
attains the infimum in $m_t(\alpha)$ for all $t \geq 1$.  However, we have that
\begin{equation*}
	\frac{1}{1} < \frac{2}{1} < \frac{\log 5}{\log 2}
\end{equation*}
meaning that $(1,1) \not\in \DD_2(\xi)$.
Nevertheless, we are able to produce the desired stronger conclusion under the additional assumption that $(a,b)$ is an upper or lower best approximation to $\xi$.

\begin{thm} \label{Main}
	Suppose $p$ and $q$ are distinct primes and $\xi = \log q/\log p$.  Assume that $(a,b)$ is an upper or lower best approximation for $\xi$ and set $\alpha = p^a/q^b$.  If
	\begin{equation*} 
		\left( \frac{p^{a_1}}{q^{b_1}}, \frac{p^{a_2}}{q^{b_2}},\ldots, \frac{p^{a_N}}{q^{b_N}}\right)
	\end{equation*}
	is a factorization which attains the infimum in $m_t(\alpha)$ for some $t > 1$ then, for every $1\leq n\leq N$, $(a_n,b_n)$ is an upper or lower best approximation for $\xi$.
\end{thm}

It follows from the work of \cite{DubSmyth2} and \cite{SamuelsCollection} that the trivial factorization $(\alpha)$ attains the infimum in $m_t(\alpha)$ for all
$t\leq 1$ provided that $\alpha\in \rat$. Consequently, Theorem \ref{Main} allows us to determine an infimum set for $\alpha$ using the upper and lower best rational approximations for $\xi$.  
We define
\begin{equation} \label{BasicInfSet}
	\mathcal B(\alpha) = \left\{ \left( \frac{p^{a_1}}{q^{b_1}},\cdots, \frac{p^{a_N}}{q^{b_N}}\right) \in \mathcal F(\alpha): (a_n,b_n)\in \GG \cap (\NN_1(\xi) \cup \DD_2(\xi)) \right\},
\end{equation}
so that $\mathcal B(\alpha)$ is the set of all factorizations of $\alpha$ which use only upper best approximations or lower best approximations for $\xi$.  

\begin{cor} \label{MainCor}
	Suppose $p$ and $q$ are distinct primes and $\xi = \log q/\log p$.  If $(a,b)$ is an upper or lower best approximation for $\xi$ then
	$\mathcal B(p^a/q^b)$ is an infimum set for $p^a/q^b$.
\end{cor}

All of the pairs $(a_n,b_n)$ appearing in \eqref{BasicInfSet} may be determined by examining the continued fraction expansion for $\log q/\log p$ in accordance with 
Theorem \ref{ApproxFrac}.  This technique will be common practice throughout the applications to be discussed in the next section.

\section{Examples, Applications and Open Questions} \label{Apps}

Let $\alpha$ be an arbitrary
algebraic number and suppose that $ A = (\alpha_1,\alpha_2,\ldots,\alpha_N)$ and $B = (\beta_1,\beta_2,\ldots,\beta_M)$ are product representations of $\alpha$.
We say that $A$ is {\it equivalent} to $B$ if $N=M$ and there exists a bijection $\sigma:\{1,2,\ldots,N\}\to \{1,2,\ldots,N\}$ such that 
$\alpha_n = \beta_{\sigma(n)}$ for all $1\leq n\leq N$.  In this case, we write $A\sim B$ and note that $\sim$ defines an equivalence relation on the set of all product 
representations of $\alpha$.
If $A = (\alpha_1,\alpha_2,\cdots,\alpha_N)$ is a product representation of $\alpha$, we define the {\it measure function of $A$} to be
the function $f_A:(0,\infty)\to [0,\infty)$ given by
\begin{equation*}
	f_A(t) = \left(\sum_{k=1}^N m(\alpha_k)^t\right)^{1/t}.
\end{equation*}
If $A\sim B$, it is obvious that $f_A  = f_B$.

\subsection{The Characteristic Transformation of $\alpha$}
In order to apply the results of Section \ref{Mains}, it will be helpful to express the elements of $\mathcal B(\alpha)$ as vectors.  
In this way, we will be able to convert the question of determining $\mathcal B(\alpha)$ into a linear algebra question.
For simplicity, we shall now write $\mathcal B(\alpha)$ to denote the set of equivalence classes of points in \eqref{BasicInfSet}.  Suppose that $p$ and $q$ are primes
with $q > p$ and set $\xi = \log q/\log p$.  Further suppose that $(a,b)$ is an upper or lower best approximation for $\xi$ and set $\alpha = p^a/q^b$.  Let
\begin{equation} \label{BAs}
	(a_1,b_1), (a_2,b_2) ,\ldots, (a_N,b_N)
\end{equation}
be the complete list of upper or lower best approximations to $\xi$ with $a_n \leq a$ and $b_n \leq b$.  Moreover, assume that $b_1 \leq b_2 \leq \cdots \leq b_N$ and if 
$b_n = b_{n+1}$ then $a_n \leq a_{n+1}$.  We note that the points listed in \eqref{BAs} depend only on $\alpha$.
The {\it characteristic transformation of $\alpha$} is defined to be the linear transformation $T_\alpha: \real^N \to \real^2$ given by
\begin{equation*}
	T_\alpha = \left( \begin{array}{cccc} a_1 & a_2 & \cdots & a_N \\ b_1 & b_2 & \cdots & b_N \end{array} \right).
\end{equation*}
It is easily verified that the rows of $T_\alpha$ are linearly independent over $\real$, which implies that $T_\alpha$ is a surjection and $\dim_\real(\ker T_\alpha) = N-2$.
A {\it factorization vector for $\alpha$} is a point $\xx = (x_1,x_2,\ldots,x_N)^T\in \real^N$ satisfying
\begin{enumerate}[(i)]
	\item $T_\alpha(\xx) = (a,b)^T$
	\item $x_n\in \nat_0$ for all $1\leq n\leq N$.
\end{enumerate}
We let $\mathcal V(\alpha)$ denote the set of all factorization vectors for $\alpha$, and for each $\xx\in \mathcal V(\alpha)$, define the {\it factorization associated to $\xx$} by
\begin{equation} \label{NewFactor}
	A_\xx = \left( \underbrace{ \frac{p^{a_1}}{q^{b_1}},\cdots,\frac{p^{a_1}}{q^{b_1}}}_{x_1\mbox{ times}},  
				 \underbrace{\frac{p^{a_2}}{q^{b_2}},\cdots,\frac{p^{a_2}}{q^{b_2}}}_{x_2\mbox{ times}}, \cdots\cdots,
				  \underbrace{\frac{p^{a_N}}{q^{b_N}},\cdots,\frac{p^{a_N}}{q^{b_N}}}_{x_N\mbox{ times}} \right).
\end{equation}
The following observation allows us to view elements of $\mathcal B(\alpha)$ as factorization vectors.

\begin{thm} \label{VectorConversion}
	If $\alpha$ satisfies the hypotheses of Theorem \ref{Main} then the map $\phi:\mathcal V(\alpha)\to \mathcal B(\alpha)$ given by $\phi(\xx) = A_\xx$
	is a bijection.
\end{thm}

If $S\subseteq \mathcal V(\alpha)$ is 
such that $\phi(S)$ is an infimum set for $\alpha$, then we simply say that $S$ is an {\it infimum set for $\alpha$}.  In this language, Corollary \ref{MainCor}
asserts that $\mathcal V(\alpha)$ is indeed an infimum set for $\alpha$.
For $\xx\in \mathcal V(\alpha)$ we write $f_\xx = f_{\phi(\xx)}$ and we say that $f_\xx$ is the {\it measure function of $\xx$}.  If $\xx = (x_1,\ldots,x_N)^T$ then
\begin{equation} \label{FactorMeasure}
	f_\xx(t) = \left(\sum_{n=1}^N x_n m\left( \frac{p^{a_n}}{q^{b_n}}\right)^t\right)^{1/t},
\end{equation}
and it follows from Theorem \ref{VectorConversion} and Corollary \ref{MainCor} that $m_t(\alpha)  = \min\left\{ f_\xx(t): \xx\in \mathcal V(\alpha)\right\}$.

\begin{ex} \label{32over27}
	Consider $\alpha = 32/27 = 2^5/3^3$.  We shall attempt to sketch the graph of $t\mapsto m_t(\alpha)$. 
	In this case, we have that $\xi = \log 3/\log 2$ so that the beginning of the continued fraction expansion for $\xi$ is given by
	\begin{equation*}
		\xi = [1;1,1,2,2,3,1,5,2,23,2,2,\ldots].
	\end{equation*}
	In view of Theorem \ref{ApproxFrac}, the first several upper and lower best approximations to $\xi$ are
	\begin{equation*}
		\{(1,0),(1,1),(2,1),(3,2),(5,3),(8,5),(11,7),(19,12),\ldots\},
	\end{equation*}
	and therefore, the characteristic transformation of $\alpha$ is given by
	\begin{equation*}
		T_\alpha = \left( \begin{array}{ccccc} 1 & 1 & 2 & 3 & 5 \\ 0 & 1 & 1 & 2 & 3 \end{array} \right).
	\end{equation*}
	By solving the system $T_\alpha(\xx) = (5,3)^T$ for $\xx$ having non-negative integer entries, we now obtain
	\begin{equation*} 
			\mathcal V(\alpha) = \left\{ \begin{pmatrix} 0 \\ 0\\ 0\\ 0\\ 1\end{pmatrix},
							\begin{pmatrix} 0 \\ 0 \\ 1 \\ 1 \\0 \end{pmatrix},
							\begin{pmatrix} 1 \\ 1 \\ 0 \\ 1\\ 0 \end{pmatrix},
							\begin{pmatrix} 0 \\ 1\\ 2 \\ 0 \\ 0 \end{pmatrix},
							\begin{pmatrix} 1 \\ 2\\ 1 \\ 0 \\ 0 \end{pmatrix},
							\begin{pmatrix} 2 \\ 3 \\ 0 \\ 0 \\ 0 \end{pmatrix}  \right\}.
	\end{equation*}
	According to Corollary \ref{MainCor}, this set forms an infimum set for $\alpha$.  In order to sketch the graph of $t\mapsto m_t(\alpha)$ we recall that 
	$m_t(\alpha)  = \min\left\{ f_\xx(t): \xx\in \mathcal V(\alpha)\right\}$ so it will be useful to obtain graphs for $f_\xx(t)$ for each $\xx\in \mathcal V(\alpha)$.  Setting
	$\xx = (x_1,x_2,x_3,x_4,x_5)^T$ and applying \eqref{FactorMeasure} we deduce that
	\begin{equation*}
		f_\xx(t) = \left( x_1m\left(\frac{2^1}{3^0}\right)^t + x_2m\left(\frac{2^1}{3^1}\right)^t + x_3m\left(\frac{2^2}{3^1}\right)^t + x_4m\left(\frac{2^3}{3^2}\right)^t
			+ x_5m\left(\frac{2^5}{3^3}\right)^t\right)^{1/t},
	\end{equation*}
	and hence,
	\begin{equation*}
		f_\xx(t) = \left( x_1(\log 2)^t + x_2(\log 3)^t + x_3(\log 4)^t + x_4(\log 9)^t + x_5(\log 32)^t\right)^{1/t}
	\end{equation*}
	The graphs of the measure functions $f_\xx(t)$, for each $\xx\in \mathcal V(\alpha)$, may be found in Figure \ref{fig:32over27full}.
	\begin{figure}[p]
		\centering
		\includegraphics[height=8cm,width=13cm]{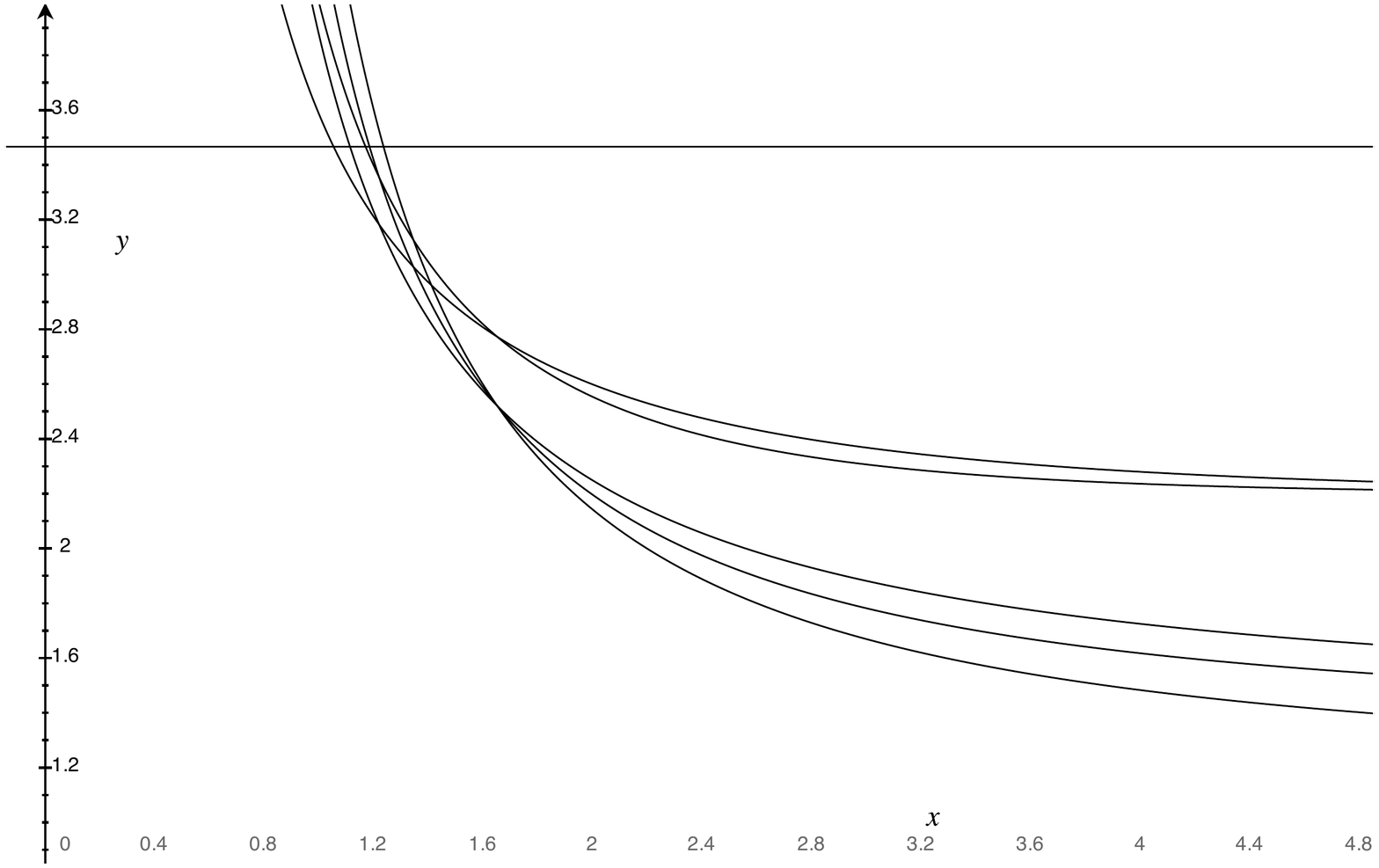}
		\caption{Measure functions of vectors in $\mathcal V(32/27)$}
		\label{fig:32over27full}
	\end{figure}
	By examining the curves in Figure \ref{fig:32over27full} we note that the set
	\begin{equation*}
		\mathcal S(\alpha) = \left\{ \begin{pmatrix} 0 \\ 0\\ 0\\ 0\\ 1\end{pmatrix},
							\begin{pmatrix} 0 \\ 0 \\ 1 \\ 1 \\0 \end{pmatrix},
							\begin{pmatrix} 0 \\ 1\\ 2 \\ 0 \\ 0 \end{pmatrix},
							\begin{pmatrix} 2 \\ 3 \\ 0 \\ 0 \\ 0 \end{pmatrix}  \right\}.
	\end{equation*}
	is also an infimum set for $\alpha$.  The graphs of the measure functions of the points in $S$ may be found in Figure \ref{fig:32over27clean}.
		\begin{figure}[p]
		\centering
		\includegraphics[height=8cm,width=13cm]{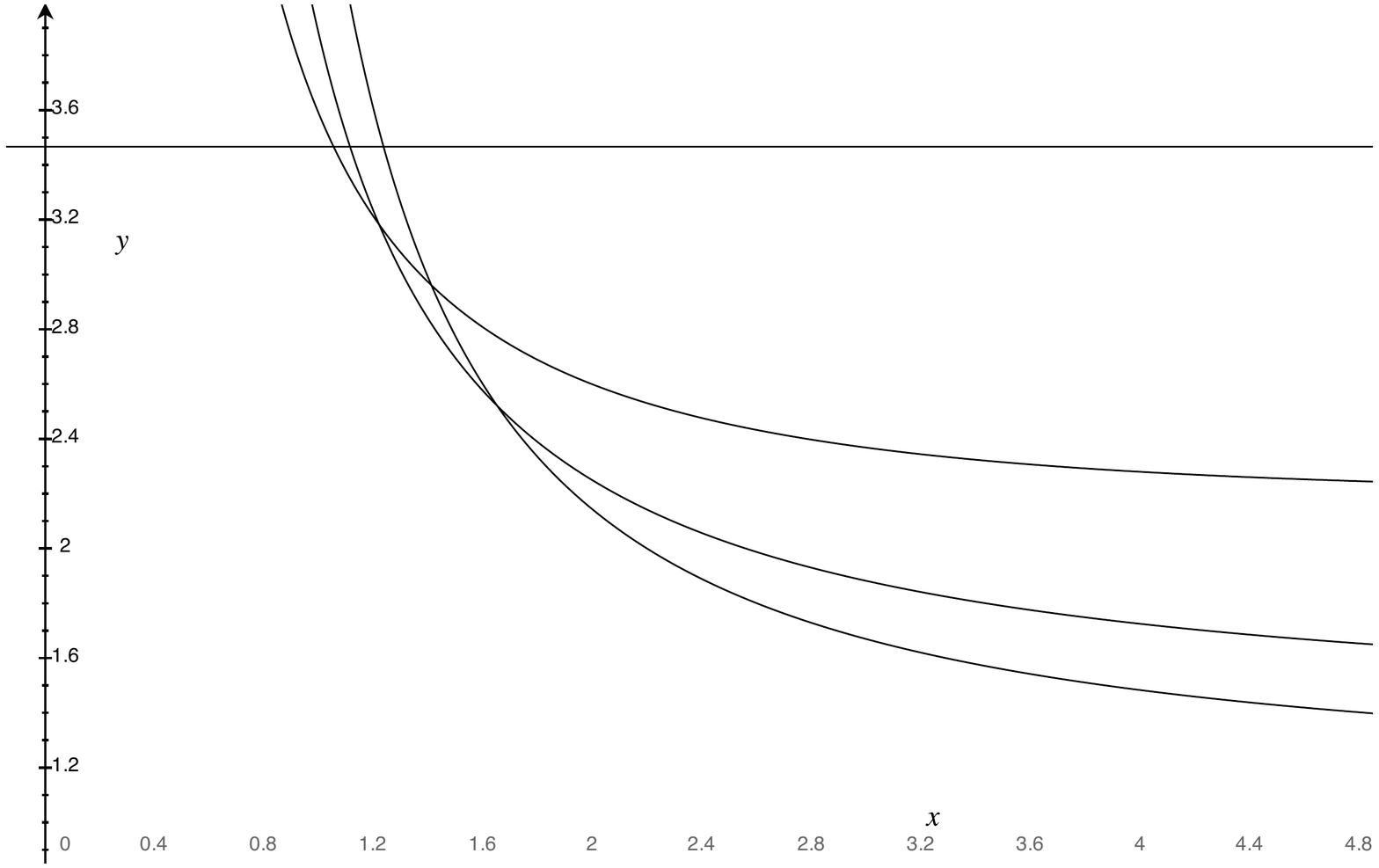}
		\caption{Measure functions of vectors in $\mathcal S(32/27)$}
		\label{fig:32over27clean}
	\end{figure}
\end{ex}

Following the definitions developed by the author in \cite{SamuelsParametrized}, we say that a positive real number $t$ is {\it standard for $\alpha$} if there exists
an open neighborhood $U$ of $t$ and $(y_1,y_2,\ldots,y_n)\in \real^N$ such that 
\begin{equation*}
	m_t(\alpha) = \left( |y_1|^t + |y_2|^t + \cdots + |y_n|^t\right)^{1/t}
\end{equation*}
for all $t\in U$.  If $t$ is not standard for $\alpha$ then we say that $t$ is {\it exceptional for $\alpha$}.  We may think of exceptional points as those points at which the infimum
attaining factorization of $\alpha$ is forced to change.  The work of \cite{SamuelsParametrized} provided an example of a rational number having two exceptional points\footnote{
This result of \cite{SamuelsParametrized} relied on a conjecture which was shown in \cite{JankSamuels} to be correct.}, however, no example having more than two exceptional
points was known.  By examining Figures \ref{fig:32over27full} and \ref{fig:32over27clean}, the example $32/27$ presented here clearly has three exceptional points and 
is the first such known example.

The method outlined in Example \ref{32over27} may be applied to other cases in order to produce examples of rational numbers having more than three exceptional points.
To obtain an example with four exceptional points, we may make a simple adjustment to Example \ref{32over27}.

\begin{ex} \label{256over243}
	Consider $\alpha = 256/243 = 2^8/3^5$.  As we are using the same primes as in Example \ref{32over27}, we obtain the same list of upper and lower best approximations
	to $\xi = \log 3/\log 2$
	\begin{equation*}
		\{(1,0),(1,1),(2,1),(3,2),(5,3),(8,5),(11,7),(19,12),\ldots\},
	\end{equation*}
	and therefore, the characteristic transformation of $\alpha$ is given by
	\begin{equation*}
		T_\alpha = \left( \begin{array}{cccccc} 1 & 1 & 2 & 3 & 5 & 8 \\ 0 & 1 & 1 & 2 & 3 & 5 \end{array} \right).
	\end{equation*}
	Once again we solve the system $T_\alpha(\xx) = (8,5)^T$ for $\xx$ having non-negative integer entries to obtain
	\begin{equation*}
		\mathcal V(\alpha) = \left\{ \begin{pmatrix} 0 \\ 0 \\ 0\\ 0\\ 0\\ 1\end{pmatrix},
							\begin{pmatrix} 0 \\ 0 \\ 0 \\ 1 \\ 1 \\0 \end{pmatrix},
							\begin{pmatrix} 0 \\ 1 \\ 1 \\ 0 \\ 1\\ 0 \end{pmatrix},
							\begin{pmatrix} 0 \\ 0 \\ 1\\ 2 \\ 0 \\ 0 \end{pmatrix},
							\begin{pmatrix} 1 \\ 2\\ 0 \\ 0 \\ 1 \\ 0  \end{pmatrix},
							\begin{pmatrix} 0 \\ 1 \\ 2 \\ 1 \\ 0 \\ 0 \end{pmatrix},
							\begin{pmatrix} 1 \\ 1 \\ 0 \\ 2 \\ 0 \\ 0 \end{pmatrix},
							\begin{pmatrix} 1 \\ 2 \\ 1 \\ 1 \\ 0 \\ 0 \end{pmatrix},
							\begin{pmatrix} 0 \\ 2 \\ 3 \\ 0 \\ 0 \\ 0 \end{pmatrix},
							\begin{pmatrix} 2 \\ 3 \\ 0 \\ 1 \\ 0 \\ 0 \end{pmatrix},
							\begin{pmatrix} 1 \\ 3 \\ 2 \\ 0 \\ 0 \\ 0 \end{pmatrix},
							\begin{pmatrix} 2 \\ 4 \\ 1 \\ 0 \\ 0 \\ 0 \end{pmatrix},
							\begin{pmatrix} 3 \\ 5 \\ 0 \\ 0 \\ 0 \\ 0 \end{pmatrix} \right\}.
	\end{equation*}
	Setting $\xx = (x_1,x_2,x_3,x_4,x_5,x_6)^T$ and applying \eqref{FactorMeasure} we deduce that
	\begin{equation*}
		f_\xx(t) = \left( x_1m\left(\frac{2^1}{3^0}\right)^t + x_2m\left(\frac{2^1}{3^1}\right)^t + x_3m\left(\frac{2^2}{3^1}\right)^t + x_4m\left(\frac{2^3}{3^2}\right)^t
			+ x_5m\left(\frac{2^5}{3^3}\right)^t + x_6m\left(\frac{2^8}{3^5}\right)^t\right)^{1/t},
	\end{equation*}
	and hence,
	\begin{equation*}
		f_\xx(t) = \left( x_1(\log 2)^t + x_2(\log 3)^t + x_3(\log 4)^t + x_4(\log 9)^t + x_5(\log 32)^t + x_6(\log 256)^t\right)^{1/t}
	\end{equation*}
	As in our previous example, we plot the graphs of the measure functions $f_\xx(t)$ for each $\xx\in \mathcal V(\alpha)$.  These graphs may be found in Figure \ref{fig:256over243full}.
	\begin{figure}[p]
		\centering
		\includegraphics[height=8cm,width=13cm]{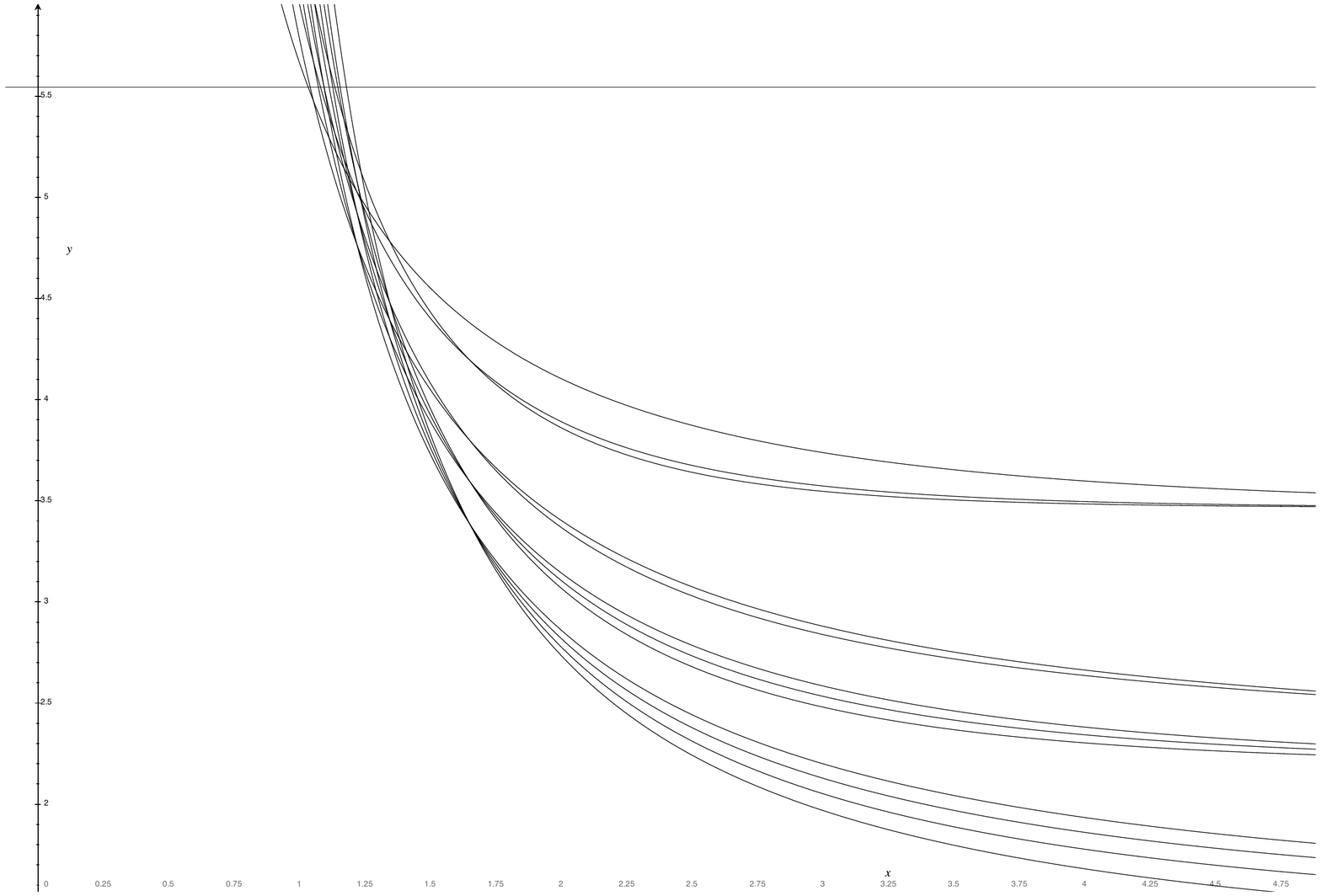}
		\caption{Measure functions of vectors in $\mathcal V(256/243)$}
		\label{fig:256over243full}
	\end{figure}
	In this case, we determine that the set
	\begin{equation*}
		\mathcal S(\alpha) = \left\{ \begin{pmatrix} 0 \\ 0 \\ 0\\ 0\\ 0\\ 1\end{pmatrix},
							\begin{pmatrix} 0 \\ 0 \\ 0 \\ 1 \\ 1 \\0 \end{pmatrix},
							\begin{pmatrix} 0 \\ 0 \\ 1\\ 2 \\ 0 \\ 0 \end{pmatrix},
							\begin{pmatrix} 0 \\ 2 \\ 3 \\ 0 \\ 0 \\ 0 \end{pmatrix},
							\begin{pmatrix} 3 \\ 5 \\ 0 \\ 0 \\ 0 \\ 0 \end{pmatrix} \right\}
	\end{equation*}
	is also an infimum set for $\alpha$ and the graphs of the these measure functions are given in Figure \ref{fig:256over243clean}.
	\begin{figure}[p]
		\centering
		\includegraphics[height=8cm,width=13cm]{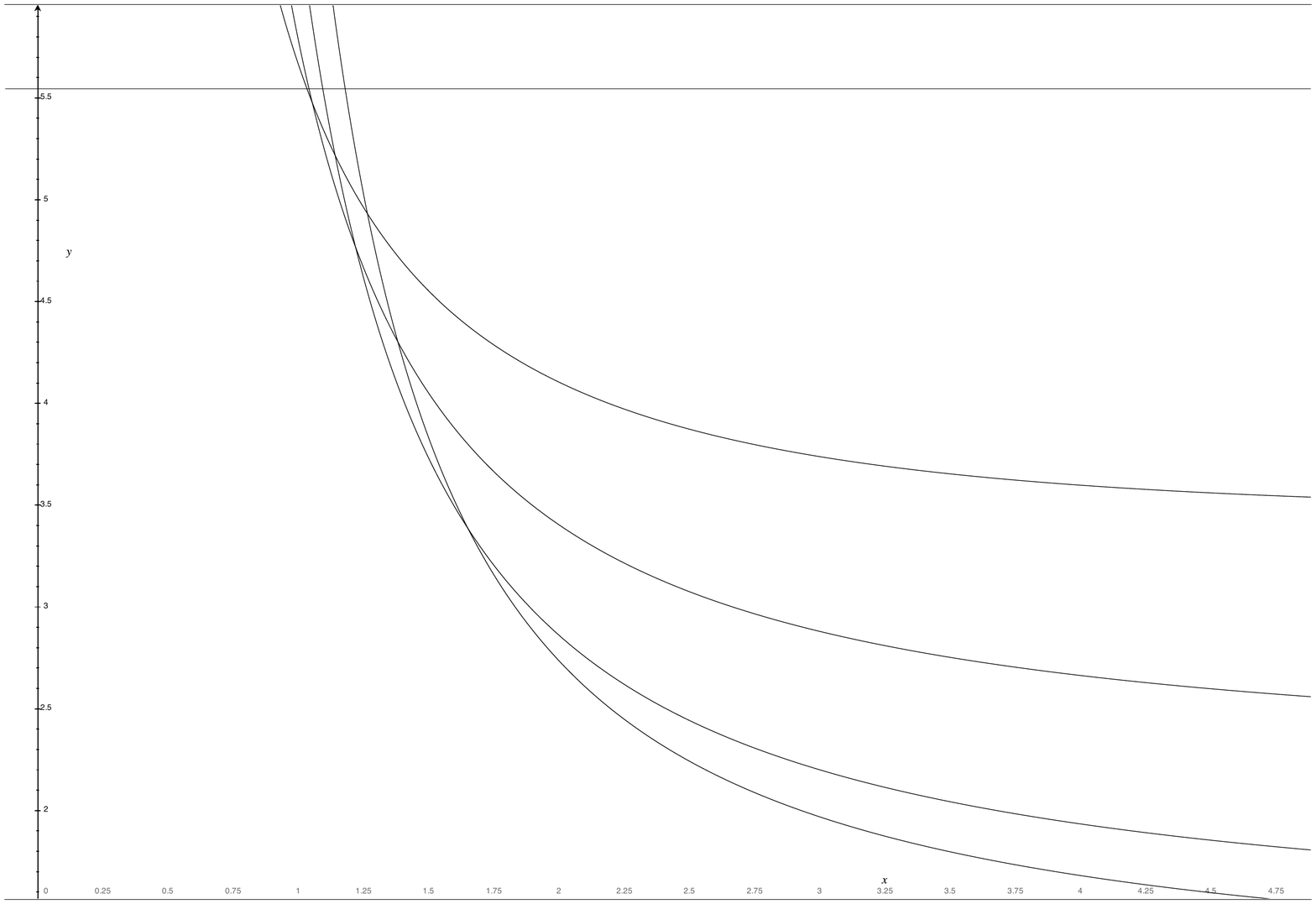}
		\caption{Measure functions of vectors in $\mathcal S(256/243)$}
		\label{fig:256over243clean}
	\end{figure}
	In this example, we observe that $\alpha$ has exactly four exceptional points.
\end{ex}

\subsection{Golden Ratio Approximations}

The collection of upper and lower best approximations to $\xi$ are particularly simple to identify when $\xi$ is close to $\phi = (1+\sqrt 5)/2$.  To see this, we let $\{h_n\}_{n=0}^\infty$
be the Fibonacci sequence, i.e., $h_0 = 0$, $h_1 = 1$ and $h_n = h_{n-1} + h_{n-2}$ for all $n\geq 2$.  It is well-known that 
\begin{equation*}
	\frac{h_2}{h_1} < \frac{h_4}{h_3} < \frac{h_6}{h_5} < \cdots < \frac{1 + \sqrt 5}{2} < \cdots < \frac{h_5}{h_4} < \frac{h_3}{h_2} < \frac{h_1}{h_0}
\end{equation*}
and that $\{h_n/h_{n-1}\}_{n=1}^\infty$ are precisely the list of convergents to $\phi$.
We now note the following important observation.

\begin{thm} \label{GRApprox}
	Suppose $n$ is a positive even integer. 
	\begin{enumerate}[(i)]
		\item\label{GoldenSqueeze}  There exist primes $p$ and $q$ such that
			\begin{equation*}
				\frac{h_n}{h_{n-1}} < \frac{\log q}{\log p} < \frac{h_{n+1}}{h_n}.
			\end{equation*}
		\item\label{CharTran} Suppose $p$ and $q$ are primes satisfying \eqref{GoldenSqueeze}.  If $\alpha = p^{h_{n+1}}/q^{h_{n}}$ 
			then the characteristic transformation of $\alpha$ is given by
			\begin{equation*}
				T_\alpha = \left( \begin{array}{cccc} h_1 & h_2 & \cdots & h_{n+1} \\ h_0 & h_1 & \cdots & h_{n} \end{array} \right).
			\end{equation*}
	\end{enumerate}
\end{thm}

In the previous subsection we studied an example satisfying Theorem \ref{GRApprox}\eqref{GoldenSqueeze} (Example \ref{32over27}).  
Indeed, taking $n=4$ we find that $p = 2$ and $q = 3$ satisfy the required property that
\begin{equation*}
	\frac{3}{2} < \frac{\log 3}{\log 2} < \frac{5}{3}.
\end{equation*}
As a result, we obtained the characteristic transformation
\begin{equation*}
	T_\alpha = \left( \begin{array}{ccccc} 1 & 1 & 2 & 3 & 5 \\ 0 & 1 & 1 & 2 & 3 \end{array} \right).
\end{equation*}
Moreover, we note that $h_n/h_{n-1} > \log 3/\log 2$ for all $n> 4$.  Consequently, $4$ is the largest value of $n$ for which the pair $(p,q) = (2,3)$ satisfies
Theorem \ref{GRApprox}\eqref{GoldenSqueeze}.  We now consider an example with $n=8$.

\begin{ex} \label{LargeExample}
	Set $n=8$ so that $h_n/h_{n-1} = 21/13$ and $h_{n+1}/h_n = 34/21$.  Note that $p=31$ and $q=257$ are primes which satisfy the required inequalities
	\begin{equation*}
		\frac{21}{13} < \frac{\log 257}{\log 31} < \frac{34}{21}.
	\end{equation*} 
	If we set $\alpha = 31^{34}/257^{21}$ then Theorem \ref{GRApprox}\eqref{CharTran} asserts that $\alpha$ has characteristic transformation given by
	\begin{equation*}
		T_\alpha = \left( \begin{array}{ccccccccc} 1 & 1 & 2 & 3 & 5 & 8 & 13 & 21 & 34 \\ 0 & 1 & 1 & 2 & 3 & 5 & 8 & 13 & 21 \end{array} \right).
	\end{equation*}
	After determining the set $\mathcal V(\alpha)$ we are able to find that
	\begin{equation*}
		\mathcal S(\alpha) =  \left\{ \begin{pmatrix} 0 \\ 0 \\ 0 \\ 0 \\ 0 \\ 0\\ 0\\ 0\\ 1\end{pmatrix},
							\begin{pmatrix} 0 \\ 0 \\ 0 \\ 0 \\ 0 \\ 0 \\ 1 \\ 1 \\0 \end{pmatrix},
							\begin{pmatrix} 0 \\ 0 \\ 0 \\ 0 \\ 0 \\ 1\\ 2 \\ 0 \\ 0  \end{pmatrix},
							\begin{pmatrix} 0 \\ 0 \\ 0 \\ 0 \\ 2\\ 3 \\ 0 \\ 0 \\ 0  \end{pmatrix},
							\begin{pmatrix} 0 \\ 0 \\ 0 \\ 3 \\ 5 \\ 0 \\ 0 \\ 0 \\ 0 \end{pmatrix},
							 \begin{pmatrix} 0 \\ 0 \\ 5 \\ 8 \\ 0 \\ 0 \\ 0 \\ 0 \\ 0 \end{pmatrix},
							 \begin{pmatrix} 0 \\ 8 \\ 13 \\ 0 \\ 0 \\ 0 \\ 0 \\ 0 \\ 0 \end{pmatrix},
							  \begin{pmatrix} 13 \\ 21 \\  0 \\ 0 \\ 0 \\ 0 \\ 0 \\ 0 \\ 0 \end{pmatrix} \right\}
	\end{equation*}
	is an infimum set for $\alpha$.  In this example, $\alpha$ has seven exceptional points.
\end{ex}

Examples \ref{32over27} and \ref{LargeExample} suggest a conjecture regarding primes satisfying Theorem \ref{GRApprox}\eqref{GoldenSqueeze}.

\begin{conj} \label{GRConjecture}
	If $p$ and $q$ are primes satisfying
	\begin{equation*}
		\frac{h_n}{h_{n-1}} < \frac{\log q}{\log p} < \frac{h_{n+1}}{h_n}
	\end{equation*}
	and $\alpha = p^{h_{n+1}}/q^{h_n}$ then 
	\begin{equation*}
		 \left\{ \begin{pmatrix} 0 \\ 0 \\ 0 \\ 0 \\ \vdots \\ 0 \\ 0\\ h_0\\ h_1 \end{pmatrix},
							\begin{pmatrix} 0 \\ 0 \\ 0 \\ 0 \\ \vdots \\ 0 \\ h_1 \\ h_2 \\0 \end{pmatrix},
							\cdots,
							\begin{pmatrix} 0 \\ h_{n-2} \\  h_{n-1} \\ 0 \\ \vdots \\ 0 \\ 0 \\ 0 \\ 0 \end{pmatrix},
							\begin{pmatrix} h_{n-1} \\ h_{n} \\  0 \\ 0 \\ \vdots \\ 0 \\ 0 \\ 0 \\ 0 \end{pmatrix}\right\}
	\end{equation*}
	is an infimum set for $\alpha$.  Moreover, $\alpha$ has $n-1$ exceptional points.
\end{conj}

Conjecture \ref{GRConjecture} is of particular interest in view of the author's work \cite{SamuelsMetric}.  Indeed, Corollary 2.3 of \cite{SamuelsMetric} asserts that every 
algebraic number has only finitely many exceptional points.  However, it is currently unknown whether there is a uniform upper bound on the number of exceptional points 
for $\alpha\in \alg$.  If Conjecture \ref{GRConjecture} holds then this question is resolved in the negative, i.e., no such uniform upper bound exists.

\subsection{The size of $\mathcal V(\alpha)$}

In the introduction, we noted that our set $\mathcal V(\alpha)$ (or equivalently $\mathcal B(\alpha)$) is smaller than the infimum set $\mathcal F(\alpha)$ 
of all factorizations of $\alpha$ described by Theorem \ref{RationalAttain}.  In the case where $\alpha$ arises from Theorem \ref{GRApprox}, we exhibit this discrepancy
by providing specific estimates on the sizes of these sets.

We first provide a general estimate on the size of $\mathcal F(\alpha)$ in the case where $\alpha = p^x/q^y$ for some $x,y\in \nat_0$.  
Let $j(x)$ and $j(y)$ denote the number of partitions of $x$ and $y$, respectively, so it follows that $\#\mathcal F(\alpha) \geq j(x)j(y)$.  
A classical result of Hardy and Ramanujan \cite{HardyRamanujan} provides an asymptotic formula for $j(x)$ which we may apply here to obtain
$\log (\#\mathcal F(\alpha)) \gg  \sqrt{x} + \sqrt{y}$.  Letting $\Omega(\alpha)$ denote the total number of not necessarily distinct prime factors of $\alpha$, we deduce that
\begin{equation} \label{FactorizationBound}
	\log (\#\mathcal F(\alpha)) \gg  \Omega(\alpha)^{1/2}.
\end{equation}
If $\alpha$ is chosen to satisfy the hypotheses of Theorem \ref{GRApprox}\eqref{CharTran} then we are able to provide an upper bound on $\mathcal V(\alpha)$ which is
significantly smaller than the right hand side of \eqref{FactorizationBound}.

\begin{thm} \label{GoldenBound}
	For each even natural number $n$, let $p_n$ and $q_n$ be primes such that
	\begin{equation*}
		\frac{h_n}{h_{n-1}} < \frac{\log q_n}{\log p_n} < \frac{h_{n+1}}{h_n}
	\end{equation*}
	and set $\alpha_n = p_n^{h_{n+1}}/q_n^{h_n}$.  Then $\log(\#\mathcal V(\alpha_n)) \ll (\log h_{n+1})^2$.
\end{thm}

For purposes of comparison with \eqref{FactorizationBound}, Theorem \ref{GoldenBound} gives
\begin{equation} \label{VBound}
	\log(\#\mathcal V(\alpha_n)) \ll (\log \Omega(\alpha_n))^2
\end{equation}
verifying that the infimum set $\mathcal V(\alpha_n)$ introduced in this article is considerably smaller than the previously best known infimum set $\mathcal F(\alpha_n)$.
We also recall that Conjecture \ref{GRConjecture} proposes the existence of an infimum set $\mathcal S(\alpha)$ having exactly $n$ elements
and as a result, we know that $\log (\#\mathcal S(\alpha_n)) \ll \log \log \Omega(\alpha_n)$.  Therefore, if Conjecture \ref{GRConjecture} holds, then we obtain an infimum set
satisfying an even stronger condition than \eqref{VBound}.

\subsection{Minimal Infimum Sets}

We recall from Corollary \ref{MainCor} and Theorem \ref{VectorConversion} that the set $\mathcal V(\alpha)$ always forms an infimum set for $\alpha$.
However, as is clear from Examples \ref{32over27}, \ref{256over243} and \ref{LargeExample}, there is often a considerably smaller set which is also an infimum set.
These observations motivate the following definitions.  

An infimum set $\mathcal X(\alpha)$ is called {\it minimal} if there does not exist a proper subset of $\mathcal X(\alpha)$ which is also an infimum set for $\alpha$.  
Using the work of the author \cite{SamuelsMetric}, it is possible to show that every algebraic number has a minimal infimum set.  Moreover, if $\mathcal X$ and $\mathcal Y$
are minimal infimum sets for $\alpha$ then there exists a bijection $\sigma:\mathcal X\to \mathcal Y$ such that $f_\xx = f_{\sigma(\xx)}$.  In particular, $\#\mathcal X = \#\mathcal Y$,
and we call this value the {\it minimality index of $\alpha$}.

It is straightforward to verify that the sets $\mathcal S(32/27)$ and $\mathcal S(256/243)$ given in Examples \ref{32over27} and \ref{256over243} 
are minimal infimum sets for $32/27$ and $256/243$, respectively.  However, there is currently no known algorithm for determining a minimal infimum set for $\alpha$ in general.
By studying the convex hull of $\mathcal V(\alpha)$, we are able to provide a partial result in this direction.  

If $S\subseteq \real^n$ is finite then the {\it convex hull of $S$} is defined to be
\begin{equation*}
	\convex(S) = \left\{ \sum_{\xx\in S} c_\xx\cdot \xx: c_\xx\geq 0\mbox{ and } \sum_{\xx\in S} c_\xx = 1\right\}.
\end{equation*}
Note that if $S$ contains exactly two points, then $\convex(S)$ is simply the line segment connecting those points.  If $\xx\in \convex(S)$ but $\xx\not \in \convex(S\setminus \{\xx\})$
then $\xx$ is called a {\it vertex} of $\convex(S)$.  We let $\vertex(S)$ denote the set of all vertices of $\convex(S)$ and note that $\vertex(S)\subseteq S$.

\begin{thm} \label{ConvexHulls}
	Suppose $p$ and $q$ are distinct primes and $\xi = \log q/\log p$.  Further assume that $(a,b)$ is an upper or lower best approximation for $\xi$ and $\alpha = p^a/q^b$.
	If $S\subseteq \mathcal V(\alpha)$ is an infimum set for $\alpha$ then $\vertex(S)$ is also an infimum set for $\alpha$.
\end{thm}

Examining our work from Example \ref{32over27}, we set $\alpha = 32/27$ and recall that 
\begin{equation*} 
		\mathcal V(\alpha) = \left\{ \begin{pmatrix} 0 \\ 0\\ 0\\ 0\\ 1\end{pmatrix},
							\begin{pmatrix} 0 \\ 0 \\ 1 \\ 1 \\0 \end{pmatrix},
							\begin{pmatrix} 1 \\ 1 \\ 0 \\ 1\\ 0 \end{pmatrix},
							\begin{pmatrix} 0 \\ 1\\ 2 \\ 0 \\ 0 \end{pmatrix},
							\begin{pmatrix} 1 \\ 2\\ 1 \\ 0 \\ 0 \end{pmatrix},
							\begin{pmatrix} 2 \\ 3 \\ 0 \\ 0 \\ 0 \end{pmatrix}  \right\}.
\end{equation*}
We observe that
\begin{equation*}
	\begin{pmatrix} 1 \\ 2\\ 1 \\ 0 \\ 0 \end{pmatrix} = \frac{1}{2}\begin{pmatrix} 0 \\ 1\\ 2 \\ 0 \\ 0 \end{pmatrix} + \frac{1}{2}\begin{pmatrix} 2 \\ 3 \\ 0 \\ 0 \\ 0 \end{pmatrix}  
\end{equation*}
proving that $(1,2,1,0,0)^T$ is not a vertex of $\mathcal V(\alpha)$, and hence, removing this vector from $\mathcal V(\alpha)$ still results in an infimum set for $\alpha$.  In fact, 
Theorem \ref{ConvexHulls} reveals that
\begin{equation*}
	\vertex(\mathcal V(\alpha)) =  \left\{ \begin{pmatrix} 0 \\ 0\\ 0\\ 0\\ 1\end{pmatrix},
							\begin{pmatrix} 0 \\ 0 \\ 1 \\ 1 \\0 \end{pmatrix},
							\begin{pmatrix} 1 \\ 1 \\ 0 \\ 1\\ 0 \end{pmatrix},
							\begin{pmatrix} 0 \\ 1\\ 2 \\ 0 \\ 0 \end{pmatrix},
							\begin{pmatrix} 2 \\ 3 \\ 0 \\ 0 \\ 0 \end{pmatrix}  \right\}
\end{equation*}
is an infimum set for $\alpha$.  Unfortunately, $(1,1,0,1,0)^T$ is a vertex of $\mathcal V(\alpha)$, so Theorem \ref{ConvexHulls} does not explain why this vector does not belong to the
minimal infimum set $\mathcal S(32/27)$ from Example \ref{32over27}.  It remains open to determine a minimal infimum set for any rational number $\alpha$ of form $\alpha = p^a/q^b$, where
$(a,b)$ is an upper or lower best approximation for $\log q / \log p$.

\section{Proofs} \label{Proofs}

 The proof of Proposition \ref{NDBasics} is sufficiently straightforward that we need not include it here.  Hence, we proceed with the proof of Proposition \ref{Relations}.

\begin{proof}[Proof of Theorem \ref{Relations}]
	To prove \eqref{Contains}, it suffices to prove that $\NN_1(\xi)\subseteq \NN_2(\xi)$ as the other set containment follows from Proposition \ref{NDBasics}.
	Suppose that $(a,b)\in \NN_1(\xi)$ but $(a,b)\not \in \NN_2(\xi)$.  Therefore, we have that
	\begin{equation*}
		\frac{a}{b} = \min\left\{ \frac{m}{n}: (m,n)\in \NN(\xi),\ m \leq a\right\} \quad\mbox{and}\quad \frac{a}{b} > \min\left\{ \frac{m}{n}: (m,n)\in \NN(\xi),\ n\leq b\right\},
	\end{equation*}
	where the inequality follows from the fact that $(a,b)\in \NN(\xi)$.  Therefore, there exists $(m,n)\in \NN(\xi)$ such that 
	\begin{equation} \label{BeatInequality}
		\frac{a}{b} > \frac{m}{n} > \xi\quad\mbox{and}\quad n\leq b.
	\end{equation}
	By our assumption that $(a,b) \in \NN_1(\xi)$, we must have that $m > a$.  We cannot have $n=0$ because then the inequalities \eqref{BeatInequality} would fail, so it follows that
	$m/n > a/n \geq a/b$, a contradiction.
	
	By Proposition \ref{NDBasics}, it is enough to prove \eqref{UpperEqual} in order to complete our proof.
	So we assume that $\xi > 1$ and that $(a,b)\in \GG\cap\NN_2(\xi)$.  If $b = 0$ then our assumption
	that $(a,b)\in \GG$ forces $a =1$ and clearly $(1,0)\in \NN_1(\xi)$.   We may assume without loss of generality that $a > 0$ and $b > 0$.
	In this situation, our assumption $(a,b)\in \NN_2(\xi)$ implies that
	\begin{equation} \label{N2Assume}
		\frac{a-1}{b} < \xi < \frac{a}{b}.
	\end{equation}
	Now suppose that $(m,n)\in \NN(\xi)$ is such that
	\begin{equation*}
		\xi < \frac{m}{n} < \frac{a}{b}.
	\end{equation*}
	Since we have assumed that $(a,b)\in \NN_2(\xi)$, we must have that $n > b$.  In order to prove that $(a,b)\in \NN_1(\xi)$ we must prove that $m > a$ which we
	shall do by contradiction.  Hence, we assume that $ m \leq a$.
	
	Since $a - 1 \geq 0$ and $n > b$ we obtain that
	\begin{equation*}
		\frac{m}{n} > \xi > \frac{a-1}{b} \geq \frac{a-1}{n},
	\end{equation*}
	and therefore $m > a-1$.  Since $m\in \nat$, we now find that $a = m$.  Consequently,
	\begin{equation*}
		\frac{m}{n} = \frac{a}{n} \leq \frac{a}{b+1} < \frac{a}{b}
	\end{equation*}
	and hence, $(a - 1)/b < a/(b+1)$.  Simplifying this inequality, we are lead to $a < b + 1$ so that $a\leq b$ contradicting our assumption that $\xi > 1$ and \eqref{N2Assume}.
\end{proof}

\subsection{Results connected to Theorem \ref{FirstAxioms}}

A point $x\in \NN(\xi)$ is called {\it reducible with respect to $\xi$} if there exist $y,z \in \NN(\xi)$ such that $x = y+z$.  Similarly, we say that a point $x\in \DD(\xi)$ is 
{\it reducible with respect to $\xi$} if there exist $y,z \in \DD(\xi)$ such that $x = y+z$.  
If $x\in\mathcal N$ is not reducible with respect to $\xi$, then we say that $x$ is {\it irreducible with respect to $\xi$}.  By applying Proposition \ref{NDBasics}, 
we find that $(a,b)$ is irreducible with respect to $\xi$ if and only if $(b,a)$ is irreducible with respect to $\xi^{-1}$.  The first step in our proof of Theorem \ref{FirstAxioms} is the following lemma.

\begin{lem} \label{BestApproxes}
	Suppose that $\xi$ is a positive real irrational number and $(a,b)\in \mathcal N$.
	\begin{enumerate}[(i)]
		\item\label{UpperIrred} Assume $(a,b)\in \NN(\xi)$.  Then $(a,b)$ is irreducible with respect to $\xi$ if and only if $(a,b)\in \GG\cap \NN_2(\xi)$.
		\item\label{LowerIrred} Assume $(a,b)\in \DD(\xi)$.  Then $(a,b)$ is irreducible with respect to $\xi$ if and only if $(a,b)\in\GG\cap \DD_1(\xi)$.
	\end{enumerate}
\end{lem}
\begin{proof}
	In view of Proposition \ref{NDBasics}, it is sufficient to prove \eqref{UpperIrred}.  We begin by assuming that $(a,b)$ is irreducible.
	If $(a,b)\not \in \GG$ then we may set $\gcd(a,b) = d > 1$, and it follows that
	\begin{equation} \label{GCDFactor}
		(a,b) = \left(\frac{a}{d},\frac{b}{d}\right) + \left(\frac{a(d-1)}{d},\frac{b(d-1)}{d}\right).
	\end{equation}
	Clearly both summands on the right hand side of \eqref{GCDFactor} belongs to $\NN(\xi)$ so that $(a,b)$ is reducible, a contradiction.
	
	We now assume that $(a,b)\not \in \NN_2(\xi)$, so there exists $(m,n)\in \NN(\xi)$ with $n\leq b$ such that 
	\begin{equation} \label{NewMin}
		\frac{a}{b} > \frac{m}{n}.
	\end{equation}
	If we had that $m\geq a$ then $m/n < a/b \leq a/n \leq m/n$, so we must have that $m < a$.  We now attempt to show that $(a-m,b-n)\in \NN(\xi)$.
	
	We cannot have $n=0$ because this contradicts \eqref{NewMin}.  Additionally, if $n=b$ then $(m,b) \in \NN(\xi)$ and the equality $ (a,b) = (a-m,0) + (m,b)$ contradicts our 
	assumption that $(a,b)$ is irreducible.  Consequently, we may assume that $0 < n < b$.  Since $(m,n)\in \NN(\xi)$, we know that $m/n > \xi$, and since $b-n > 0$, we may
	multiply both sides of this inequality by $b-n$.  Now we obtain that
	\begin{equation*}
		\frac{m}{n}(b-n) > \xi(b-n)
	\end{equation*}
	which leads to
	\begin{equation*}
		\frac{mb}{n} - m > \xi(b-n)
	\end{equation*}
	Now using \eqref{NewMin} and the fact that $b > 0$, we obtain that
	\begin{equation*}
		\frac{ab}{b} - m  > \xi(b-n).
	\end{equation*}
	Hence, we have shown that $(a-m,b-n)\in \NN(\xi)$.  Therefore, $(a,b) = (a-m,b-n) + (m,n)$ satisfies the requirement to be reducible, a contradiction.
	
	Next, we assume that $(a,b)\in \GG\cap \NN_2(\xi)$.  If $b = 0$ then since $\gcd(a,b) = 1$, we know that $ a= 1$.  Clearly $(1,0)$ is irreducible, so we may assume that $b\ne 0$.
	Suppose now that $(a,b)$ is reducible.  Therefore, there exist $(a_1,b_1),(a_2,b_2)\in \NN(\xi)$ such that
	\begin{equation}  \label{FactoringA}
		a = a_1 + a_2\quad \mbox{and}\quad  b = b_1 + b_2.  
	\end{equation}
	We cannot have that $a_1 = 0$ or $a_2 = 0$ because $(0,k)\not \in \NN(\xi)$ for any $k\in\nat_0$.  Moreover, we cannot have $b_1 = 0$ because then $b_2\ne 0$ and 
	\begin{equation*}
		\frac{a}{b}  = \frac{a_1 + a_2}{b_2} > \frac{a_2}{b_2}
	\end{equation*}
	contradicting our assumption that $(a,b)\in \NN_2(\xi)$.  
	Similarly, we cannot have that $b_2=0$.  In view of these remarks, we may now assume that none of $a_1,a_2,b_1,b_2$ are equal to $0$.
	
	We observe that $b_1,b_2 < b$, and since $(a,b)\in \NN_2(\xi)$, we obtain that
	\begin{equation*}
		\frac{a}{b} \leq \frac{a_1}{b_1} \quad\mbox{and}\quad \frac{a}{b} \leq \frac{a_2}{b_2}.
	\end{equation*}
	Applying \eqref{FactoringA}, we deduce that
	\begin{equation*}
		\frac{a_1 + a_2}{b_1 + b_2} \leq \frac{a_1}{b_1}\quad\mbox{and}\quad \frac{a_1 + a_2}{b_1 + b_2} \leq \frac{a_2}{b_2}.
	\end{equation*}
	Simplifying these inequalities, we are lead to 
	\begin{equation*}
		a_2b_1 \leq a_1b_2\quad\mbox{and}\quad a_1b_2 \leq a_2b_1,
	\end{equation*}
	and therefore
	\begin{equation*}
		\frac{a_1}{b_1} = \frac{a_2}{b_2}.
	\end{equation*}
	Now assume that $k,\ell\in \nat$ are such that $\gcd(k,\ell) = 1$ and 
	\begin{equation*}
		\frac{k}{\ell} = \frac{a_1}{b_1} = \frac{a_2}{b_2}.
	\end{equation*}
	Hence, there exist $c_1,c_2\in \nat$ such that $a_1 = c_1k$, $b_1 = c_1\ell$, $a_2 = c_2k$ and $b_2 = c_2\ell$, and it follows from \eqref{FactoringA} that
	\begin{equation*}
		a = k(c_1 + c_2)\quad\mbox{and}\quad b = \ell(c_1 + c_2)
	\end{equation*}
	contradicting our assumption that $\gcd(a,b) = 1$.
\end{proof}

Suppose $\alpha$ is a positive rational different than $1$ and
\begin{equation*} \label{AlphaGeneralFact}
	A = \left(\frac{r_1}{s_1},\frac{r_2}{s_2},\ldots,\frac{r_N}{s_N}\right)\in (\rat^\times)^N
\end{equation*}
is a factorization of $\alpha$.  We say that $A$ is {\it biased} if either
\begin{equation*}
	r_n > s_n\mbox{ for all } 1 \leq n \leq N\qquad\mbox{or}\qquad r_n < s_n\mbox{ for all } 1 \leq n \leq N.
\end{equation*}
In the former case, we say that $A$ is {\it numerator biased}, and in the latter case, we say that $A$ is {\it denominator biased}.  

Now suppose that $\alpha = p^a/q^b$, where $(a,b)\in \mathcal N$ and $p$ and $q$ are distinct primes.
In this case, every factorization of $\alpha$ must have the form
\begin{equation*} 
	A = \left( \frac{p^{a_1}}{q^{b_1}}, \frac{p^{a_2}}{q^{b_2}},\ldots, \frac{p^{a_N}}{q^{b_N}}\right),
\end{equation*}
where $a = a_1 + \cdots + a_N$ and $b = b_1 + \cdots + b_N$.  
The biased factorizations of $p^a/q^b$ are closely connected to the sets $\GG$, $\NN_2(\xi)$ and $\DD_1(\xi)$.

\begin{lem} \label{TotalBias}
	Let $p$ and $q$ be distinct primes, and set $\xi = \log q/\log p$.  If $(a,b)\in \NN(\xi)$ then the following conditions are equivalent.
	\begin{enumerate}[(i)]
		\item\label{NoBias} $p^a/q^b$ has a no non-trivial biased factorization.
		\item\label{Irreducible} $(a,b)$ is irreducible with respect to $\xi$.
		\item\label{Approxes} $(a,b) \in \GG\cap \NN_2(\xi)$.
	\end{enumerate}
	Similarly, if $(a,b)\in \DD(\xi)$ then the following conditions are equivalent.
	\begin{enumerate}[(i)]
		\item\label{NoBias2} $p^a/q^b$ has a no non-trivial biased factorization.
		\item\label{Irreducible2} $(a,b)$ is irreducible with respect to $\xi$.
		\item\label{Approxes2} $(a,b)\in \GG\cap \DD_1(\xi)$.
	\end{enumerate}
\end{lem}
\begin{proof}
	In view of Proposition \ref{NDBasics} it is sufficient to prove the first statement of the lemma.  Moreover, it follows from Lemma \ref{BestApproxes} that 
	\eqref{Irreducible} $\iff$ \eqref{Approxes}.  Hence, we need only show that \eqref{NoBias} $\iff$ \eqref{Irreducible}.  
	
	If we assume that $(a,b)$ is reducible with respect to $\xi$,
	there exist $(a_1,b_1),(a_2,b_2)\in \NN(\xi)$ such that $(a,b) = (a_1,b_1) + (a_2,b_2)$.  Since $(a_i,b_i) \in \NN(\xi)$, we know that $a_i/b_i > \log q/\log p$ for 
	$i\in \{1,2\}$ which implies that
	\begin{equation*}
		p^{a_1} > q^{b_1}\quad\mbox{and}\quad p^{a_2} > q^{b_2}.
	\end{equation*}
	Hence, $(p^{a_1}/q^{b_1}, p^{a_2}/q^{b_2})$ is a non-trivial biased factorization of $p^a/q^b$.
	
	If we assume that $p^a/q^b$ has a non-trivial biased factorization $(p^{a_1}/q^{b_1}, p^{a_2}/q^{b_2})$ then certainly $a_1 + a_2 = a$ and $b_1+b_2 = a$.
	We also have that
	\begin{equation*}
		p^{a_1} > q^{a_2}\quad\mbox{and}\quad p^{a_2} > q^{b_2}.
	\end{equation*}
	which means that $(a_i,b_i)\in \NN(\xi)$ implying that $(a,b)$ is reducible.
\end{proof}

In order to complete the proof of Theorem \ref{FirstAxioms}, we must establish a connection between non-trivial biased factorizations of $p^a/q^b$ and 
the infimum in $m_t(\alpha)$.

\begin{lem} \label{Biased}
	Suppose that $\alpha$ is a positive rational number, $A = (r_1/s_1,\ldots,r_N/s_N)$ is a factorization of $\alpha$, and $t\in (1,\infty)$.
	If there exists $n$ such that $r_n/s_n$ has a non-trivial biased factorization, then $A$ cannot attain the infimum in $m_t(\alpha)$.
\end{lem}
\begin{proof}
	Suppose that $r_n/s_n$ has a non-trivial biased factorization and assume without loss of generality that $r_n > s_n$.  Therefore, there exist
	$w,x,y,z\in\nat$ such that $w>x$, $y > z$ and
	\begin{equation*}
		\frac{r_n}{s_n} = \frac{w}{x}\cdot \frac{y}{z},
	\end{equation*}
	where $\gcd(w,x) = \gcd(w,z) = \gcd(y,x) = \gcd(y,z) = 1$.  It follows that $r_n = wy$ and $s_n = xz$.  These facts combine to ensure that
	\begin{equation*}
		m\left(\frac{r_n}{s_n}\right) = \log r_n = \log w + \log y = m\left( \frac{w}{x}\right) + m\left( \frac{y}{z}\right) 
	\end{equation*}
	Since $t> 1$ we now obtain that
	\begin{equation*}
		m\left(\frac{r_n}{s_n}\right)^t = \left(m\left( \frac{w}{x}\right) + m\left( \frac{y}{z}\right) \right)^t > m\left( \frac{w}{x}\right)^t + m\left( \frac{y}{z}\right)^t.
	\end{equation*}
	Moreover, we clearly have that
	\begin{equation*}
		\alpha = \left(\prod_{\substack{ m = 1 \\ m \ne n}}^N \frac{r_n}{s_n}\right)\cdot \frac{w}{x}\cdot \frac{y}{z}
	\end{equation*}
	so that $A$ cannot attain the infimum in the definition of $m_t(\alpha)$.
\end{proof}

Theorem \ref{FirstAxioms} now follows easily by applying Lemmas \ref{TotalBias} and \ref{Biased}.

\subsection{Results Connected to Theorem \ref{Main}}

We now wish to complete the proof of Theorem \ref{Main}.  For this purpose, we need three additional definitions and some lemmas.
Assume that $\xi$ is a positive real irrational number.  A point $(a,b)\in \mathcal N$ is called an {\it upper boundary point for $\xi$} if $(a,b)\in \NN(\xi)$ and $(a,b+1)\in\DD(\xi)$.
Similarly, $(a,b)$ is called a {\it lower boundary point for $\xi$} if $(a,b)\in \DD(\xi)$ and $(a+1,b)\in\NN(\xi)$.  In either case, $(a,b)$ will simply be called a {\it boundary point}.

\begin{lem} \label{BoundaryBA}
	Assume that $\xi$ is a positive real irrational number and $(a,b)\in \GG$.  Then the following conditions hold.
	\begin{enumerate}[(i)]
		\item\label{UBP} Assume $(a,b)\in \NN_2(\xi)$.  Then $(a,b)$ is an upper boundary point for $\xi$ if and only if $(a,b)\in \NN_1(\xi)$
		\item\label{LBP} Assume $(a,b)\in \DD_1(\xi)$.  Then $(a,b)$ is a lower boundary point for $\xi$ if and only if $(a,b)\in \DD_2(\xi)$.
	\end{enumerate}
\end{lem}
\begin{proof}
	By Proposition \ref{NDBasics} it is sufficient to prove either one of the assertions in the lemma.  We shall prove \eqref{LBP}.  
	In this case, assume that $(a,b)\in \DD_1(\xi)$. 
	
	First suppose that $(a,b)$ is a lower boundary point for $\xi$ and that $(m,n)\in \DD(\xi)$ satisfies
	\begin{equation*}
		\frac{a}{b} < \frac{m}{n} < \xi.
	\end{equation*}
	Since $(a,b)\in \DD_1(\xi)$, we know that $a< m$.  In order to establish that $(a,b)\in \DD_2(\xi)$, we must prove that $b < n$ which we shall do by contradiction.
	Hence, assume that $b \geq n$.
	
	If $a = 0$ then since $\gcd(a,b) = 1$ we obtain that $(a,b) = (0,1)$ so that $n = 1$.  This means that
	\begin{equation*}
		\frac{a}{b} < \frac{1}{1} \leq \frac{m}{n} < \xi
	\end{equation*}
	which contradicts our assumption that $(a,b)$ is a lower boundary point for $\xi$.

	Therefore, we now assume that $a > 0$.  If $b > n$ then
	\begin{equation*}
		\frac{a}{b} < \frac{a}{n} < \frac{m}{n} < \xi
	\end{equation*}
	so that $(a,n)\in \DD(\alpha)$ contradicting our assumption that $(a,b)\in\DD_1(\xi)$.  If $b = n$ then
	\begin{equation*}
		\frac{a}{b} < \frac{m}{b} = \frac{m}{n} < \alpha.
	\end{equation*}
	But since $a < a + 1 \leq m$, we obtain 
	\begin{equation*}
		\frac{a}{b} < \frac{a+1}{b} \leq \frac{m}{n} < \alpha
	\end{equation*}
	contradicting our assumption that $(a,b)$ is a lower boundary point for $\xi$.  This proves that $b < n$ so it follows that $(a,b)\in \DD_2(\xi)$.
	
	For the converse, if $(a,b)$ is a not lower boundary point for $\xi$ then
	\begin{equation*}
		\frac{a}{b} < \frac{a+1}{b} < \xi
	\end{equation*}
	contradicting our assumption that $(a,b)\in \DD_2(\xi)$.
\end{proof}

We obtain an important lemma regarding boundary points in optimal factorizations.

\begin{lem} \label{BoundaryPointFactorization}
	Suppose $p$ and $q$ are distinct primes and $\xi = \log q/\log p$.  Assume that $(a,b)$ is a boundary point for $\xi$ and that
	\begin{equation} \label{InfAttFact}
		\left( \frac{p^{a_1}}{q^{b_1}}, \frac{p^{a_2}}{q^{b_2}},\ldots, \frac{p^{a_N}}{q^{b_N}}\right)
	\end{equation}
	attains the infimum in $m_t(p^a/q^b)$ for some $t > 0$.  Then $(a_n,b_n)$ is a boundary point for all $n$.
\end{lem}
\begin{proof}
	Assume without loss of generality that $(a,b) \in \DD(\xi)$ so that $(a,b)$ is a lower boundary point for $\xi$.  This means that $(a,b)\in \DD(\xi)$ and $(a+1,b)\in \NN(\xi)$.
	Now assume that there exists $n$ such that $(a_n,b_n)$ is not a boundary point and we distinguish two cases.
	
	We first assume that $(a_n,b_n)\in \DD(\xi)$.  Since $(a_n,b_n)$ is not a lower boundary point, we know that $(a_n+1,b_n)\in \DD(\xi)$.  
	As a result we have that
	\begin{equation} \label{SameMeasures}
		m\left( \frac{p^{a_n}}{q^{b_n}}\right)  = m\left( \frac{p^{a_n+1}}{q^{b_n}}\right).
	\end{equation}
	If $(a_m,b_m)\in \DD(\xi)$ for all $1 \leq m\leq N$ then we would have that
	\begin{equation*}
		\frac{p^{a+1}}{q^b} = \left(\prod_{\substack{ m=1 \\ m\ne n}}^N \frac{p^{a_m}}{q^{b_m}} \right)\cdot \frac{p^{a_n + 1}}{q^{b_n}} < 1
	\end{equation*}
	contradicting our assumption that $(a+1,b)\in \NN(\xi)$.  Therefore, we assume that there exists $k$ such that $p^{a_k}/q^{b_k} > 1$ so that
	\begin{equation} \label{SmallerMeasures}
		m\left( \frac{p^{a_k}}{q^{b_k}}\right)  > m\left( \frac{p^{a_k-1}}{q^{b_k}}\right).
	\end{equation}
	However, we clearly have that
	\begin{equation*}
		\frac{p^a}{q^b} = \left(\prod_{\substack{ m=1 \\ m\not\in \{n,k\}}}^N \frac{p^{a_m}}{q^{b_m}} \right)\cdot \frac{p^{a_n + 1}}{q^{b_n}} \cdot \frac{p^{a_k-1}}{q^{b_k}}.
	\end{equation*}
	Combining this with \eqref{SameMeasures} and \eqref{SmallerMeasures}, we contradict our assumption that \eqref{InfAttFact} attains the infimum in $m_t(p^a/q^b)$ for some $t>0$.
	
	Next, we assume that $(a_n,b_n)\in \NN(\xi)$ so that $(a_n,b_n)$ is not an upper boundary point.  This means that $(a_n,b_n+1)\in \NN(\xi)$ and
	\begin{equation*}
		m\left( \frac{p^{a_n}}{q^{b_n}}\right)  = m\left( \frac{p^{a_n}}{q^{b_n+1}}\right)
	\end{equation*}
	We cannot have that $p^{a_m}/q^{b_m} > 1$ for all $m$ because then
	\begin{equation*}
		\frac{p^a}{q^b} = \prod_{m=1}^N \frac{p^{a_m}}{q^{b_m}} > 1
	\end{equation*}
	contradicting our assumption that $(a,b)\in \DD(\xi)$.  Therefore, there exists $k$ such tat $p^{a_k}/q^{b_k} < 1$ and it follows that
	\begin{equation*}
		m\left( \frac{p^{a_k}}{q^{b_k}}\right)  > m\left( \frac{p^{a_k}}{q^{b_k-1}}\right) 
	\end{equation*}
	once again contradicting our assumption that \eqref{InfAttFact} attains the infimum in $m_t(p^a/q^b)$ for some $t>0$.
\end{proof}

We may now quickly record the proof of Theorem \ref{Main} as a consequence of Theorem \ref{FirstAxioms} and Lemma \ref{BoundaryBA}.

\begin{proof}[Proof of Theorem \ref{Main}]
	 Assume without loss of generality that $(a,b)$ is an upper best approximate for $\xi$ and set $\alpha = p^a/q^b$.  Lemma \ref{BoundaryBA} implies that
	 $(a,b)$ must be an upper boundary point for $\xi$.  Further suppose that 
	\begin{equation*} 
		\left( \frac{p^{a_1}}{q^{b_1}}, \frac{p^{a_2}}{q^{b_2}},\ldots, \frac{p^{a_N}}{q^{b_N}}\right)
	\end{equation*}
	is a factorization attaining the infimum in $m_t(\alpha)$ for some $t > 1$.  If $(a_n,b_n)\in \NN(\xi)$ then Theorem \ref{FirstAxioms} implies $(a_n,b_n)\in \GG\cap \NN_2(\xi)$.
	By Lemma \ref{BoundaryPointFactorization}, $(a_n,b_n)$ is an upper boundary point, and then by Lemma \ref{BoundaryBA}
	$(a_n,b_n)\in \GG\cap\NN_1(\xi)$.  If $(a_n,b_n)\in \DD(\xi)$ we similarly obtain that $(a_n,b_n)\in\GG\cap \DD_2(\xi)$ as required.
\end{proof}

\subsection{Results connected to applications of our main results}

Our ability to rewrite factorizations of $p^a/q^b$ as vectors (Theorem \ref{VectorConversion}) is a critical component of our applications discussed above.  We begin this subsection by
providing its proof.

\begin{proof}[Proof of Theorem \ref{VectorConversion}]
	Suppose that $\{(a_n,b_n)\}_{n=1}^\infty$ are the upper and lower best approximations of $\log q/\log p$ with $b_1 \leq b_2 \leq \ldots$.  
	Also, we may assume that $a_1\leq a_2 \leq \ldots$.  Since $\alpha$ satisfies the hypotheses of Theorem \ref{Main}, there exists $N\in \nat$ such that
	$\alpha = p^{a_N}/q^{b_N}$ and hence, the characteristic transformation of $\alpha$ is given by
	\begin{equation*}
		T_\alpha = \left( \begin{array}{cccc} a_1 & a_2 & \cdots & a_N \\ b_1 & b_2 & \cdots & b_N \end{array} \right).
	\end{equation*}
	
	Now assume that $A\in \mathcal B(\alpha)$.  Since $A$ is factorization of $\alpha$, all entries of $\alpha$ must have the form $p^{a_n}/q^{b_n}$ for $1\leq n \leq N$, and therefore,
	$A$ is equivalent to the factorization 
	\begin{equation*}
		\left( \underbrace{ \frac{p^{a_1}}{q^{b_1}},\cdots,\frac{p^{a_1}}{q^{b_1}}}_{x_1\mbox{ times}},  
				 \underbrace{\frac{p^{a_2}}{q^{b_2}},\cdots,\frac{p^{a_2}}{q^{b_2}}}_{x_2\mbox{ times}}, \cdots\cdots,
				  \underbrace{\frac{p^{a_N}}{q^{b_N}},\cdots,\frac{p^{a_N}}{q^{b_N}}}_{x_N\mbox{ times}} \right)
	\end{equation*}
	for some $x_1,x_2,\ldots,x_N\in \nat_0$.  Moreover, we note that
	\begin{equation*}
		\sum_{n=1}^N x_na_n = a_N \quad\mbox{and}\quad \sum_{n=1}^N x_nb_n = a_N
	\end{equation*}
	so that $(x_1,x_2,\ldots,x_N)^T\in \mathcal V(\alpha)$ and it follows that $\phi$ is surjective.  Finally, the injectivity of $\phi$ is easily verified completing the proof.
\end{proof}

Before proceeding with the proof of Theorem \ref{GRApprox} we note an additional fact regarding the continued fraction expansion of an irrational number $\xi>1$. 
For simplicity, if $(a,b)$ is an upper (or lower) best approximation to $\xi$, we shall say that $a/b$ is an upper (or lower) best approximation to $\xi$.  If
$\xi = [x_0;x_1,x_2,x_3,\ldots]$ then we have already noted in Theorem \ref{ApproxFrac} that $a/b\in \rat$ is an upper or lower best approximation 
to $\xi$ if and only if there exists $n\in \nat$ and $1\leq x\leq x_n$ such that
\begin{equation*}
	\frac{a}{b} = [x_0;x_1,x_2,\ldots,x_{n-1},x].
\end{equation*}
If $n$ is even, then we observe two additional useful facts:
\begin{enumerate}[(i)]
	\item  $[x_0;x_1,x_2,\ldots,x_{n-1},x]$ is strictly increasing as a function of $x$.
	\item  $[x_0;x_1,x_2,\ldots,x_{n-1},x_n]\in \mathcal L(\xi)$
\end{enumerate}
Still under the assumption that $n$ is even, we conclude that $[x_0;x_1,x_2,\ldots,x_{n-1},x]$ is a lower best approximation to $\xi$ for all $1\leq x\leq x_n$.
Analogously, if $n$ is odd then $[x_0;x_1,x_2,\ldots,x_{n-1},x]$ is an upper best approximation to $\xi$ for all $1\leq x\leq x_n$.
We are now prepared to provide our proof of Theorem \ref{GRApprox}.

\begin{proof}[Proof of Theorem \ref{GRApprox}]
	Let $x = h_n/h_{n-1}$ and $x + \varepsilon = h_{n+1}/h_n$ so that $\varepsilon >0$.  Now select $p$ so that $p^\varepsilon \geq 2$.  It is a well-known result of Chebyshev that there 
	exists a prime $q$ such that $p^x \leq q \leq 2p^x \leq p^{x+\varepsilon}$, and it follows that $x < \log q/\log p < x+ \varepsilon$ proving \eqref{GoldenSqueeze}.\footnote{
	We thank Paul Fili for providing the proof of Theorem \ref{GRApprox}\eqref{GoldenSqueeze}.}
	
	Now let $\xi = \log q / \log p$ and let $\phi = (1+\sqrt 5)/2$.  We first claim that $h_{n}/h_{n-1}$ is a lower best approximation to $\xi$.  To see this, we assume that 
	$a/b\in \rat$ is such that
	\begin{equation*}
		\frac{h_n}{h_{n-1}} < \frac{a}{b} < \xi\quad\mbox{and}\quad b > h_{n-1}.
	\end{equation*}
	Moreover, from \eqref{GoldenSqueeze} we deduce that either
	\begin{equation} \label{Ordering}
		\frac{h_n}{h_{n-1}} < \xi < \phi < \frac{h_{n+1}}{h_n}\quad\mbox{or}\quad \frac{h_n}{h_{n-1}} < \phi < \xi < \frac{h_{n+1}}{h_n}.
	\end{equation}
	In both cases we conclude that
	\begin{equation*}
		\left| \frac{h_n}{h_{n-1}} - \phi\right| > \left| \frac{a}{b} - \phi\right|,
	\end{equation*}
	contradicting the known fact that $h_n/h_{n-1}$ is a best approximation to $\phi$.
	
	If $n=2$ then $h_{n}/h_{n-1} = 1$ and $h_{n+1}/h_n = 3/2$.  In this situation,  the continued fraction expansion for $\xi$ must have $1$ as its initial entry and the desired result
	follows from Theorem \ref{ApproxFrac}.  In every other case, there are precisely two continued fraction expansions for $h_n/h_{n-1}$ given by
	\begin{equation*}
		\frac{h_n}{h_{n-1}} = [\underbrace{1;1,1,\ldots,1,1}_{n-1\ \mbox{times}}] = [\underbrace{1;1,1,\ldots,1,1}_{n-3\ \mbox{times}},2].
	\end{equation*}
	Since $h_n/h_{n-1}$ is a lower best approximation to $\xi$, Theorem \ref{ApproxFrac} implies that the continued fraction expansion for $\xi$
	has the form
	\begin{equation*}
		\xi =  [\underbrace{1;1,1,\ldots,1,1}_{n-2\ \mbox{times}},a_{n-2},a_{n-1},a_n\ldots]\quad\mbox{or}\quad
		\xi = [\underbrace{1;1,1,\ldots,1,1}_{n-3\ \mbox{times}},a_{n-3},a_{n-2},a_{n-1},\ldots],
	\end{equation*}
	where in the latter case $a_{n-3} \geq 2$.  In the latter case, we also deduce that
	\begin{equation*}
		[\underbrace{1;1,1,\ldots,1,1}_{n-3\ \mbox{times}},2] = \frac{h_n}{h_{n-1}} \quad\mbox{and}\quad 
		[\underbrace{1;1,1,\ldots,1,1}_{n-2\ \mbox{times}}] = \frac{h_{n-1}}{h_{n-2}}
	\end{equation*}
	both belong to $\mathcal L(\xi)$ which contradicts the fact that $h_n/h_{n-1} < \phi < h_{n+1}/h_n < h_{n-1}/h_{n-2}$.  This forces
	\begin{equation} \label{XiFrac}
		\xi =  [\underbrace{1;1,1,\ldots,1,1}_{n-2\ \mbox{times}},a_{n-2},a_{n-1},a_n\ldots]
	\end{equation}
	and the result follows again by applying Theorem \ref{ApproxFrac}.
\end{proof}

Although it isn't necessary for the proof of Theorem \ref{GRApprox}\eqref{CharTran}, we note that the continued fraction expansion for $\xi$ given in \eqref{XiFrac} must have $a_{n-1} = 1$.
That is, we know that
\begin{equation*}
	\xi =  [\underbrace{1;1,1,\ldots,1,1}_{n-1\ \mbox{times}},a_{n-1},a_{n},\ldots].
\end{equation*}
Otherwise, $h_n/h_{n-1}$ and $h_{n+1}/h_n$ must both belong to $\mathcal L(\xi)$ which contradicts our assumption from Theorem \ref{GRApprox}\eqref{GoldenSqueeze}.

We now establish the proof of our estimate on the size of $\mathcal V(\alpha)$ when $\alpha$ satisfies Theorem \ref{GRApprox}.

\begin{proof}[Proof of Theorem \ref{GoldenBound}]
	By Theorem \ref{GRApprox}, we conclude that $\alpha_n$ has characteristic transformation
	\begin{equation*}
		T_\alpha = \left( \begin{array}{cccc} h_1 & h_2 & \cdots & h_{n+1} \\ h_0 & h_1 & \cdots & h_{n} \end{array} \right).
	\end{equation*}
	Assuming $\xx = (x_1,x_2,\ldots,x_n,x_{n+1})^T\in \mathcal V(\alpha_n)$ we know that $x_i \leq h_{n+1}$ for all $i$.  Otherwise, we would have
	\begin{equation*}
		x_1h_1 + x_2h_2 + \cdots x_nh_n + x_{n+1}h_{n+1} \geq x_ih_{i} > h_{n+1},
	\end{equation*}
	a contradiction.  Hence, it follows that $\#\mathcal V(\alpha) \leq (h_{n+1}+1)^{n+1}$.  In addition, using basic facts about the Fibonacci sequence, we find that 
	$ h_{n+1} \geq 2^{n/2}$ so that
	\begin{equation*}
		n \leq \frac{2\log h_{n+1}}{\log 2}.
	\end{equation*}
	These observations yield
	\begin{equation*}
		\log(\#\mathcal V(\alpha_n)) = (n+1)\log (h_{n+1} + 1) \leq \left(\frac{2\log h_{n+1}}{\log 2} + 1\right)\log (h_{n+1} + 1),
	\end{equation*}
	and we obtain $\log(\#\mathcal V(\alpha_n)) \ll (\log h_{n+1})^2$ as required.
\end{proof}

Although it is somewhat less important for our applications, the use of convex hulls in Theorem \ref{ConvexHulls} can help us to provide a small infimum set for $p^a/q^b$.
We record the proof of this result here.

\begin{proof}[Proof of Theorem \ref{ConvexHulls}]
	It is well-known that $\convex(S) = \convex(\vertex(S))$ so that
	\begin{equation*}
		\convex(S) = \left\{ \sum_{\xx\in \vertex(S)} c_\xx\cdot \xx: c_\xx\geq 0\mbox{ and } \sum_{\xx\in \vertex(S)} c_\xx = 1\right\}.
	\end{equation*}
	If $\yy\in S\setminus \vertex(S)$ we may assume that $\xx_1,\xx_2,\ldots,\xx_k\in \vertex(S)$ and $c_1,c_2,\ldots,c_k\in (0,1)$ are such that
	\begin{equation} \label{Hull}
		\yy = c_1\xx_1 + c_2\xx_2 + \cdots + c_k\xx_k\quad\mbox{and}\quad c_1 + c_2 + \cdots + c_k = 1.
	\end{equation}
	For each positive real number $t$, we may select $\zz_t \in \{\xx_1,\xx_2,\ldots,\xx_k\}$ so that 
	\begin{equation} \label{MinVectors}
		f_{\zz_t}(t) = \min\{f_{\xx_1}(t),f_{\xx_2}(t),\ldots,f_{\xx_k}(t)\}.
	\end{equation}
	To complete the proof of the theorem, it is sufficient to show that $f_\yy(t) \geq f_{\zz_t}(t)$ for all $t\in (0,\infty)$.  
	
	To see this, we write $\yy = (y_1,y_2,\ldots,y_N)$ and $\xx_i = (x_{i1},x_{i2},\ldots,x_{iN})$ for all $1\leq i\leq k$ and assume that $\alpha$ has characteristic transformation
	\begin{equation*}
		T_\alpha = \left( \begin{array}{cccc} a_1 & a_2 & \cdots & a_N \\ b_1 & b_2 & \cdots & b_N \end{array} \right).
	\end{equation*}
	Now observe that
	\begin{equation*}
		f_\yy(t)^t = \sum_{n=1}^N y_n m\left( \frac{p^{a_n}}{q^{b_n}}\right)^t = \sum_{n=1}^N\sum_{i=1}^k c_i\xx_{in}m\left(\frac{p^{a_n}}{q^{b_n}}\right)^t.
	\end{equation*}
	Reversing the order of summation on the right hand side yields
	\begin{equation*}
		f_\yy(t)^t = \sum_{i=1}^k c_i \sum_{n=1}^N \xx_{in}m\left(\frac{p^{a_n}}{q^{b_n}}\right)^t = \sum_{i=1}^k c_i f_{\xx_i}(t)^t.
	\end{equation*}
	Now applying \eqref{Hull} and \eqref{MinVectors}, we obtain
	\begin{equation*}
		f_\yy(t)^t  \geq \sum_{i=1}^k c_i f_{\zz_t}(t)^t = f_{\zz_t}(t)^t
	\end{equation*}
	completing the proof.
\end{proof}

\end{document}